\documentclass[10pt,a4paper]{article}

\usepackage{amsfonts}
\usepackage{graphicx}
\usepackage[latin1]{inputenc}
\usepackage{amsmath,amssymb,amsthm}
\usepackage{calc}
\usepackage{bbm}
\usepackage{cite}

\newtheorem{theorem}{Theorem}
\newtheorem{lemma}[theorem]{Lemma}
\newtheorem{corollary}[theorem]{Corollary}
\newtheorem{proposition}[theorem]{Proposition}

\usepackage{bbm}
\usepackage{amsmath}
\usepackage{ifthen}
\usepackage{calc}
\usepackage{mathrsfs}

\usepackage{color}
\definecolor{light}{gray}{.91}

\newlength{\fboxsepdefault}
\newlength{\fboxsepnew}
\newlength{\SatzboxTextwidth}
\newlength{\GrauBoxTextwidth}
\newlength{\HauptsatzboxTextwidth}
\newlength{\hfillparboxwidth}
\providecommand{\dup}{{\ensuremath{{\operatorname{d}}}}}

 \providecommand{\C}{{\ensuremath{\mathbf{C}}}}

 \providecommand{\E}{{\ensuremath{\mathbbm{E}}}}

 \providecommand{\N}{{\ensuremath{\mathbbm{N}}}}

 \renewcommand{\P}{{\ensuremath{\mathbbm{P}}}}

 \providecommand{\R}{{\ensuremath{\mathbbm{R}}}}

 \providecommand{\1}{{\ensuremath{\mathbbm{1}}}}

 \providecommand{\CP}{{\ensuremath{\cal P}}}
 \providecommand{\CQ}{{\ensuremath{\cal Q}}}

 \providecommand{\MCG}{{\ensuremath{\mathcal G}}}

 \providecommand{\th} {{\ensuremath{\hat{t}}}}

 \providecommand{\Ft}    {{\ensuremath{\tilde{F}}}}

 \providecommand{\St}    {{\ensuremath{\tilde{S}}}}

 \providecommand{\Zt}    {{\ensuremath{\tilde{Z}}}}

 \providecommand{\Fb}    {{\ensuremath{\bar{F}}}}

 \providecommand{\Sb}    {{\ensuremath{\bar{S}}}}

 \providecommand{\Xb}    {{\ensuremath{\bar{X}}}}
 
 \providecommand{\Zb}    {{\ensuremath{\bar{Z}}}}

 \providecommand{\nb} {{\ensuremath{\bar{n}}}}

\providecommand{\phit}{\ensuremath{\tilde{\phi}}}%

\providecommand{\ro}[1]    {{{(#1)}}}
\providecommand{\rob}[1]   {{{\bigl(#1\bigr)}}}
\providecommand{\robb}[1]  {{{\biggl(#1\biggr)}}}
\providecommand{\roB}[1]   {{{\Bigl(#1\Bigr)}}}

\providecommand{\sqb}[1]  {{{\bigl[#1\bigr]}}}
\providecommand{\sqbb}[1] {{{\biggl[#1\biggr]}}}
\providecommand{\sqB}[1]  {{{\Bigl[#1\Bigr]}}}

\providecommand{\absbb}[1]{{\ensuremath{\Bigl|#1\Bigr|}}}

\providecommand{\floor}[1]  {{\ensuremath{\lfloor#1\rfloor}}}

\providecommand{\al}      {{\ensuremath{\alpha}}}

\renewcommand{\th}    {{\ensuremath{\theta}}}
\providecommand{\ld}      {{\ensuremath{\lambda}}}

\providecommand{\eps}     {{\ensuremath{\varepsilon}}}

\providecommand{\limn}   {{\ensuremath{{\displaystyle \lim_{n \ra \infty}}}}}

\providecommand{\limt}   {{\ensuremath{{\displaystyle \lim_{t \ra \infty}}}}}

\providecommand{\limT}   {{\ensuremath{{\displaystyle \lim_{T \ra \infty}}}}}

\providecommand{\qqasn}  {\ensuremath{\qquad\text{as }n\to\infty}}

\providecommand{\ra}{\rightarrow}

\providecommand{\wlimeps}%
      {{\ensuremath{\stackrel{\eps \rightarrow \infty}%
                            {\Longrightarrow}}}}

\providecommand{\varwlim}  {\xrightarrow{\text{w}}}
\providecommand{\varwlimn}  {\xrightarrow[n  \ra\infty]{\text{w}}}
\providecommand{\varwlimt}  {\xrightarrow[t  \ra\infty]{\text{w}}}

\providecommand{\fa}{\ensuremath{\;\;\forall\;}}

\providecommand{\Gen}{{\ensuremath{\MCG}}}

\providecommand{\Genb}{{\ensuremath{\bar{\Gen}}}}

\providecommand{\del}{{\ensuremath{\partial}}}

\providecommand{\clearemptydoublepage}%
    {\newpage{\pagestyle{empty}\cleardoublepage}}

\newlength{\mylen}

\providecommand{\eqd}{{\ensuremath{\;\overset{d}{=}\;}}}

\def\mathclapinternal#1#2{\clap{$\mathsurround=0pt#1{#2}$}}
\def\clap#1{\hbox to 0pt{\hss#1\hss}}

\def\mathclap{\mathpalette\mathclapinternal}

\makeatletter
\newlength{\Breit@}\newlength{\breit@}\newlength{\diff@renz}
\providecommand{\St@ckrel}[2]{%
  \settowidth{\Breit@}{\ensuremath{^{#1}}}
  \settowidth{\breit@}{\ensuremath{#2}}
  \ifthenelse{\Breit@>\breit@}{
    \setlength{\diff@renz}{(\Breit@-\breit@)/2}
    \hspace{\diff@renz}
  }{}}

\providecommand{\etStackrel}[2]{%
  \St@ckrel{#1}{#2}%
  &\stackrel{\mathclap{#1}}{#2}
  \St@ckrel{#1}{#2}}
\makeatother

\renewcommand{\E}{\mathbb{E}}
\renewcommand{\P}{\mathbb{P}}
  \newcommand{\arcosh}{\operatorname{arcosh}}

\begin{document}
\author{
Christian B\"oinghoff\footnote{Research supported by the German Research Foundation (DFG) and the 
Russian Foundation of Basic Research (Grant DFG-RFBR 08-01-91954)}
\ and
Martin Hutzenthaler\footnote{Research supported by the Institute for Mathematical Sciences of the National University of Singapore}\\[2mm]
\emph{University of Frankfurt}\/ and
\emph{University of Munich (LMU)}
}

\title{Branching diffusions in random environment}
\maketitle
%
\makeatletter
\let\@makefnmark\relax
\let\@thefnmark\relax
\@footnotetext{\emph{AMS 2010 subject classification:} 60J80; 60K37, 60J60}
\@footnotetext{\emph{Key words and phrases:}
Branching process, random environment, diffusion approximation,
Laplace transform,
survival probability,
backbone construction,
immortal individual,
ultimate survival
  }
\makeatother
\normalsize\rm
\begin{abstract}
  We consider the diffusion approximation of branching processes in random environment
  (BPREs).
  This diffusion approximation is similar to  and mathematically more tractable than BPREs.
  We obtain the exact asymptotic behavior of the survival probability.
  As in the case of BPREs, there is a phase transition in the subcritical regime
  due to different survival opportunities.
  In addition, we characterize the process conditioned to never go extinct
  and establish a backbone construction.
  In the strongly subcritical regime,
  mean offspring numbers
  are increased but still subcritical
  in the process conditioned to never go extinct.
  Here survival is solely due to an immortal individual,
  whose offspring are the ancestors of additional families.
  In the weakly subcritical regime, the mean offspring number
  is supercritical in the process conditioned to never go extinct.
  Thus this process survives with positive probability
  even if there was no immortal individual.
\end{abstract}

%
%
%
\section{Introduction and main results}%
\label{sec:Introduction}
Branching processes in random environment (BPREs)
have been introduced by Smith and
Wilkinson (1969)\nocite{SmithWilkinson1969}
(see also Smith (1968)\nocite{Smith1968})
and have attracted
considerable interest in the last decade
(e.g.~\cite{
AfanasyevEtAl2011pre,
AfanasyevEtAl2005AOP,
AfanasyevEtAl2005SPA,
BansayeBerestycki2009,
BoeinghoffEtAl2010,
Dyakonova2008,
GeigerKersting2002,
GeigerEtAl2003,
Kozlov2006,
Vatutin2004,
WangFang1999}).
On the one hand this is due to the more realistic model
compared with classical branching processes.
On the other hand this is due to interesting properties
such as phase transitions in the subcritical regime.
Here we consider the diffusion approximation of BPREs which can be viewed
as continuous mass branching process in random environment.
Our results are qualitatively analogous to discrete mass BPREs.
The main observation of this article is that the diffusion approximation of BPREs
is a simple model ($3$ parameters)
having explicit formulas for various expressions.
In particular, the contributions of the branching process and of
the environment are explicit in terms of the parameters.
These properties make the diffusion approximation interesting for applications.

The diffusion approximation of BPREs has been conjectured by
Keiding (1975)\nocite{Keiding1975} and has been established
by Kurtz (1978)\nocite{Kurtz1978}.
This diffusion approximation is the strong
solution $(Z_t,S_t)_{t\geq0}$ of the stochastic differential equations (SDEs)
\begin{equation}  \begin{split}  \label{eq:BDRE}
 dZ_t&=\frac 12 \sigma_e^2 Z_t dt+Z_t dS_t+\sqrt{\sigma_b^2 Z_t}dW^{(b)}_t\\
 dS_t&=\alpha dt +\sqrt{\sigma_e^2}dW^{(e)}_t
\end{split}     \end{equation}
for $t\geq0$
where $S_0=0$.
The parameters satisfy $\al\in\R,\sigma_e\in[0,\infty)$ and $\sigma_b\in(0,\infty)$.
The processes $(W^{(b)}_{t})_{t\geq0}$
and $(W^{(e)}_{t})_{t\geq0}$ are independent standard Brownian motions.
We denote the process $(Z_t,S_t)_{t\geq0}$ as
  \textit{branching diffusion in random environment} (BDRE).
Moreover, we will refer to $(S_t)_{t\geq0}$ as the \textit{associated Brownian motion},
which is a non-standard Brownian motion.
To be accurate, the conjecture of
Keiding (1975)\nocite{Keiding1975} did not include 
the term $\frac{1}2\sigma_e^2 Z_t\,dt$ which is a characteristic part
for random environment.
Moreover, Helland (1981)\nocite{Helland1981}
gave an inaccurate ``proof'' of N.~Keiding's conjecture.
So we attribute the correct statement of the diffusion approximation
of BPREs to Kurtz (1978)\nocite{Kurtz1978}.
The BDRE has not been studied since Kurtz (1978).
For this reason, we first discuss properties of the BDRE obtained
by Kurtz (1978) beginning with the diffusion approximation.

First we introduce BPREs in order to state the diffusion approximation.
Our formulation follows the notation of
Afanasyev et al.~\cite{AfanasyevEtAl2005AOP}.
Let $\Delta$ be the Polish space of probability measures
on $\N_0:=\{0,1,2,\ldots\}$ equipped with the metric of total
variation.
Fix $n\in\N:=\{1,2,\ldots\}$ for the moment.
Let $\Pi^{(n)}=\rob{Q_0^{(n)},Q_1^{(n)},\ldots}$
be a sequence of independent and identically distributed random
variables taking values in $\Delta$.
Conditioned on $\Pi^{(n)}$ the BPRE $(Z_i^{(n)})_{i\in\N_0}$
is defined recursively through
\begin{equation}  \label{eq:BPRE}
  Z^{(n)}_{i+1}:=\sum_{j=1}^{Z_{i}^{(n)}}\xi_{j,i}^{(n)},\qquad i\in\N_0,
\end{equation}
where $Z_0^{(n)}$ is independent of $\Pi^{(n)}$
and where $(\xi_{j,i}^{(n)})_{j,i\in\N_0}$ conditioned on $\Pi^{(n)}$
are independent random variables with distribution
\begin{equation}
  \P\rob{\xi_{j,i}^{(n)}=k|\Pi^{(n)}}= Q_i^{(n)}(k),
  \qquad\fa j,i,k\in\N_0.
\end{equation}
Let the mean of the environment at time $i\in\N_0$ be defined through
\begin{equation}
  m(Q^{(n)}_i):=\sum_{k=0}^\infty k \,Q_i^{(n)}(k).
\end{equation}
Define a continuous time version of the BPRE
through $Z_t^{(n)}:=Z_{\floor{t}}^{(n)}$ 
where $\floor{t}:=\max\{m\in\N_0\colon m\leq t\}$
for every $t\in[0,\infty)$.
The associated random  walk $\ro{S_t^{(n)}}_{t\geq0}$ is defined
through
\begin{equation}  \label{eq:random.walk}
  S_t^{(n)}:=\sqrt{n}\sum_{i=0}^{\floor{t}-1}\log\rob{m\rob{Q_{i}^{(n)}}},
  \qquad t\in[0,\infty).
\end{equation}
This random walk is central for the BPRE as it determines the mean
of the BPRE:
\begin{equation}  \label{eq:expectation_BPRE}
  \E\sqb{Z_t^{(n)}|\Pi^{(n)}}
  =\E\sqb{Z_0^{(n)}}\prod_{i=0}^{\floor{t}-1}m\rob{Q_{i}^{(n)}}
  =\E\sqb{Z_0^{(n)}}\exp\roB{\frac{S_t^{(n)}}{\sqrt{n}}},
  \qquad t\in[0,\infty).
\end{equation}
We included the factor $\sqrt{n}$ in the definition of the
associated random walk to have the usual scaling in the limit as $n\to\infty$.

Next we let $n\to\infty$ to obtain the diffusion approximation.
The following assumptions mainly ensure that the associated
random walk converges to a Brownian motion with infinitesimal
drift $\al\in\R$
and infinitesimal standard deviation $\sigma_e\in[0,\infty)$:
\begin{align}
 \label{eq:assumption.mean}
  \limn n\cdot\E\sqB{m(Q_0^{(n)})-1}&=\alpha\in\R\\
 \label{eq:assumption.var.e}
  \limn n\cdot\E\sqB{\rob{m(Q_0^{(n)})-1}^2}&=\sigma_e^2\in[0,\infty)\\
 \label{eq:assumption.third.moment}
  \sup_{n\in\N} \E\sqbb{
     \sum_{k=0}^\infty\absbb{\frac{k}{m\big(Q_0^{(n)}\big)}-1}^3\,\cdot Q_0^{(n)}(k)
                       }
          &<\infty.
\end{align}
So the branching process is near-critical as
$m(Q_0^{(n)})\to 1$ in distribution as $n\to\infty$.
If $\sigma_e>0$, then the environment comprises
both supercritical and subcritical phases.
In addition, we suppose that
\begin{equation}  \begin{split} \label{eq:assumption.var.b}
  \limn \E\sqB{\sum_{k=0}^\infty
        \Big(\frac{k}{m\big(Q_0^{(n)}\big)}-1\Big)^2
        Q_0^{(n)}(k)}
  =\sigma_b^2\in(0,\infty).
\end{split}     \end{equation}
So $\al$ is a parameter of expected super-/subcriticality,
$\sigma_e$ is a parameter for the standard deviation of the offspring
mean around the critical value $1$ and $\sigma_b^2$ is the
mean offspring variance per individual per generation.
Under the above assumptions, 
Corollary 2.18
of
Kurtz (1978)\nocite{Kurtz1978}
implies
that the suitably rescaled
BPRE converges in distribution to a diffusion.
\begin{proposition} \label{p:diffusion.approximation}
  Assume that $Z_0^{(n)}/n\to Z_0$ in distribution as $n\to\infty$.
  Under the assumptions~\eqref{eq:assumption.mean},
  \eqref{eq:assumption.var.e},
  \eqref{eq:assumption.third.moment}
  and
  \eqref{eq:assumption.var.b} we have that
  \begin{equation} \label{eq:diffusion.approximation}
    \roB{\frac{Z_{tn}^{(n)}}{n},\frac{S_{tn}^{(n)}}{\sqrt{n}}}_{t\geq0}
    \varwlimn\rob{Z_t,S_t}_{t\geq0}
  \end{equation}
  in the Skorohod topology (see e.g.~\cite{EthierKurtz1986}) where the limiting diffusion is
  the strong solution of the SDEs~\eqref{eq:BDRE}.
\end{proposition}
\noindent
The assumptions of Corollary 2.18 of Kurtz (1978)\nocite{Kurtz1978} are
checked in Section~\ref{sec:diffusion.approximation}.

Inserting the random environment 
$(S_t)_{t\geq0}$ into the diffusion equation~\eqref{eq:BDRE}
of the BDRE, we see that $(Z_t)_{t\geq0}$ solves the stochastic
differential equation
\begin{equation}  \label{eq:Z}
 dZ_t=\Bigl(\alpha+\frac 12 \sigma_e^2\Bigr)Z_t\, dt
     +\sqrt{\sigma_e^2 Z_t^2}\,dW^{(e)}_t+\sqrt{\sigma_b^2 Z_t}\,dW^{(b)}_t
\end{equation}
for $t\in[0,\infty)$.
Comparing with Feller's branching diffusion (i.e.~\eqref{eq:Z} with $\sigma_e=0$)
there are two differences.
First there is an additional drift term $\tfrac{1}{2}\sigma_e^2 Z_t dt$.
Second there is an additional diffusion term $\sigma_e Z_t dW^{(e)}_t$.
Both terms originate in the conditional expectation
\begin{equation}  \label{eq:expectation_BDRE}
  \E\sqb{Z_t|S_t}
  =\E\sqb{Z_0}\exp\rob{S_t},
  \qquad t\in[0,\infty),
\end{equation}
almost surely, which is a geometric Brownian motion
and solves the SDE
\begin{equation}
  dY_t= \Bigl(\alpha+\frac 12 \sigma_e^2\Bigr)Y_t dt
     +\sigma_e Y_t \,dW^{(e)}_t,\quad Y_0=\E\sqb{Z_0},
\end{equation}
for $t\geq0$.

Now we come to properties of the BDRE which embody the
branching property conditioned on the environment.
Theorem 2.10 of Kurtz (1978) implies that the BDRE is in fact a reweighted
and time-changed branching diffusion (see Lemma~\ref{l:time.change.immigration} below
for a different proof):
\begin{proposition}   \label{p:time.change}
  Assume $\al\in\R$, $\sigma_b\in(0,\infty)$ and $\sigma_e\in[0,\infty)$.
  Let
  $(W_t^{(b)})_{t\geq0}$
  and
  $(W_t^{(e)})_{t\geq0}$
  be independent standard Brownian motions.
  Let $(F_t)_{t\geq0}$ be the strong solution of
  \begin{equation}
    dF_t=\sqrt{F_t}dW_t^{(b)}
  \end{equation}
  for $t\in[0,\infty)$
  and let $S_t:=\al t+\sigma_e W_t^{(e)}$ for $t\in[0,\infty)$.
  Moreover define $(\tau(t))_{t\geq0}$ through
  \begin{equation}
    \tau(t):=\int_0^t e^{-S_s}\sigma_b^2\,ds
  \end{equation}
  for $t\in[0,\infty)$.
  Then
  \begin{equation}  \label{eq:time.change}
    \left( F_{\tau(t)}e^{S_t}, S_t \right)_{t\geq 0}
  \end{equation}
  is a weak solution of~\eqref{eq:BDRE}, that is, 
  is a version of the BDRE~\eqref{eq:BDRE}.
\end{proposition}
\noindent
Due to this property, many results on Feller's branching diffusion
carry over to the BDRE~\eqref{eq:BDRE}. For example,
$(Z_te^{-S_t})_{t\geq 0}$ is a time-changed Feller branching diffusion
and is therefore infinitely divisible.
Another simple implication of Proposition~\ref{p:time.change} is an explicit formula 
for the Laplace transform
of the BDRE~\eqref{eq:BDRE} conditioned on the environment,
which has not been reported yet.
We agree on the convention that
\begin{equation}
  \frac{c}{0}:=\begin{cases}
                 \infty&\text{if }c\in(0,\infty]\\
                 0&\text{if }c=0
               \end{cases} \ ,
  \quad
  \frac{c}{\infty}:=0 \text{ for }c\in[0,\infty)
  \quad\text{and that }0\cdot\infty=0.
\end{equation}
Throughout the paper, the notation $\P^{z}$ and $\E^{z}$ refers to the starting
point of the involved process, e.g. $\P^z(Z_t\in\cdot):=\P(Z_t\in\cdot|Z_0=z)$
for $z\in[0,\infty)$
or  $\P^{(z,s)}((Z_t,S_t)\in\cdot):=\P((Z_t,S_t)\in\cdot|(Z_0,S_0)=(z,s))$
for $z\in[0,\infty)$ and $s\in\R$.
\begin{corollary} \label{c:Laplace.transform}
  Assume $\al\in\R$, $\sigma_b\in(0,\infty)$ and $\sigma_e\in[0,\infty)$.
  Let $(Z_t,S_t)_{t\geq0}$ be the strong solution of~\eqref{eq:BDRE}.
  Then we have that
  \begin{equation}  \begin{split}  \label{eq:Laplace.transform}
      \E^{(z,0)}\Bigl[ \exp\big(-\lambda Z_t\big)|\rob{S_s}_{s\leq t}\Bigr]
    =
    \exp\bigg(-\frac{z}{\int_0^t \frac{\sigma_b^2}{2}
         \exp\big(-S_s \big) ds
         + \frac{1}{\lambda} \exp(-S_t)
         }
         \bigg)
    \quad
  \end{split}     \end{equation}
  for all $t,z,\ld\in[0,\infty)$ almost surely.
\end{corollary}
\noindent
The proof is deferred to Section~\ref{sec:Laplace.transform}.
If $\sigma_e=0$, then~\eqref{eq:Laplace.transform}
is just the Laplace transform of
Feller's branching diffusion with criticality parameter $\al$ and
branching rate $\sigma_b^2$.

The simplicity of the right-hand side of~\eqref{eq:Laplace.transform}
derives from the fact that the distribution of the integral of
the squared geometric Brownian motion with drift $\beta\in\R$,
\begin{align}  \label{eq:At}
 A_t^{(\beta)}:=\int_{0}^t \exp\big(2(\beta s+W^{(e)}_s)\big)ds, \quad t\in[0,\infty),
\end{align}
is well understood
(see e.g.~\cite{MatsumotoYor2003,Dufresne2001,Yor1992,ComtetEtAl1998}).
Even more, the density of the joint distribution
of $(A_t^{(\beta)},W^{(e)}_t+\beta t)$ is known rather explicitly
for every $t\in(0,\infty)$.
Define
\begin{equation} \label{eq:joint.density}
  a_t(x,u)\,du:=\P(A_t^{(\beta)}\in du|W_t^{(e)}+\beta t=x)
\end{equation}
for $t,u\in(0,\infty)$ and $x\in\R$.
Then the density
of $(A_t^{(\beta)},W^{(e)}_t+\beta t)$
satisfies that
\begin{equation}  \label{eq:joint.density.explicit}
  \frac{1}{\sqrt{2\pi t}}\exp\roB{-\frac{x^2}{2t}}a_t(x,u)
  =\frac{1}{u}\exp\roB{-\frac{1}{2u}(1+e^{2x})}\theta_{e^x/u}(t)
\end{equation}
where
\begin{equation}
  \theta_r(t)=\frac{r}{\sqrt{2\pi^3 t}}
       \exp\roB{\frac{\pi^2}{2t}}
     \int_0^\infty \exp\rob{-\frac{y^2}{2t}}\exp\rob{-r\,\cosh(y)}
     \,\sinh(y)\sin\roB{\frac{\pi y}{t}}\,dy
\end{equation}
for all $t,u,r\in(0,\infty)$ and $x\in\R$,
see Proposition $2$ of Yor (1992)\nocite{Yor1992}.
Using the explicit formula~\eqref{eq:joint.density.explicit} allows to answer
rather fine questions by elementary (but sometimes nontrivial)
calculations.
Thereby the BDRE~\eqref{eq:BDRE} becomes one of the most tractable
processes in the class of BPREs.

The following corollary of Corollary~\ref{c:Laplace.transform} provides
an explicit expression for the survival probability.
Define the parameter $\beta\in[-\infty,\infty]$
and the function $f\colon[0,\infty]\to[0,1]$ through
\begin{equation} \label{deff}
   \beta:= -\frac{2\alpha}{\sigma^2_e}
  \quad\text{and}\quad
  f(x):=1-\exp\Big(-\frac{\sigma_e^2}{\sigma_b^2}\cdot x\Big),\qquad
  x\in[0,\infty],
\end{equation}
if $\sigma_e\neq0$ and through $\beta=0$ and $f\equiv 0$ if $\sigma_e=0$.
\enlargethispage{\baselineskip}
\begin{corollary} \label{c:survival_prob}
  Assume $\al\in\R$, $\sigma_b\in(0,\infty)$ and $\sigma_e\in[0,\infty)$.
  Let $(Z_t,S_t)_{t\geq0}$ be the strong solution of~\eqref{eq:BDRE}.
  Then
 \begin{equation}  \begin{split}
 \P^{(z,0)}\Bigl(Z_t>0\,\big|\,(S_s)_{s\leq t}\Bigr)
   &=1-\exp\biggl(-\frac{z}{\int_0^t\frac{\sigma_b^2}{2}
               \exp\bigl(-S_s\bigr)ds}\biggr)
 \end{split} \label{survprob}    \end{equation}
 for every $t\in(0,\infty)$ and every $z\in[0,\infty)$ almost surely.
 If $\beta>-1$ and $\sigma_e>0$, then
 \begin{equation} \label{survneu} 
 \P^z(Z_t>0)
  =\E\Big[f\Big(\frac{z}{2A^{(\beta)}_{t\sigma_e^2/4}}\Big)\Big]
  =\int_0^{\infty}f(za) p_{t\sigma_e^2/4,\beta}(a)\, da
 \end{equation}
 for every $t\in(0,\infty)$ and every $z\in[0,\infty)$ where
 the density function of $1/(2A_v^{(\beta)})$ satisfies
  \begin{equation}  \begin{split} \label{densitity_general}
    \lefteqn{
    p_{v,\beta}(a)da
    := \P\Big(\frac{1}{2 A^{(\beta)}_v} \in da\Big)
    }\\
    &= \frac{e^{-\beta^2 v/2} e^{\pi^2/2v}}{\sqrt{2}\pi^2 \sqrt{v}}
       \Gamma\Big(\frac{\beta+2}{2}\Big) e^{-a} a^{-(\beta+1)/2}
       \int_0^\infty\int_0^\infty e^{-\xi^2/2v}  s^{(\beta-1)/2}e^{-as}
       \frac{\sinh(\xi)\cosh(\xi)\sin(\pi \xi/v)}{(s+(\cosh(\xi))^2)^{\frac{\beta+2}{2}}}\, d\xi\, ds\,da
  \end{split}     \end{equation}
  on $(0,\infty)$ for every $v\in(0,\infty)$.
\end{corollary}
\noindent
The proof is deferred to Section~\ref{sec:Laplace.transform}.

The asymptotic behavior of the survival probability strongly depends on $\al$.
As in the case of classical branching processes, the survival probability
stays positive, converges to zero polynomially fast or
converges to zero exponentially fast according to whether the
process is \emph{supercritical} ($\al>0$), \emph{critical} ($\al=0$)
or \emph{subcritical} ($\al<0$), respectively.
Now in case of a random environment it is known for BPREs that there
is another phase transition in the subcritical regime.
For BDREs this phase transition turns out to occur at $\al=-\sigma_e^2$.
We adopt the standard notation of the literature on BPREs
for the different regimes and say that
the BDRE is   \emph{weakly subcritical} if
$-\sigma_e^2<\al<0$,
\emph{intermediately subcritical} if
$\al=-\sigma_e^2$
and
\emph{strongly subcritical} if
$\al<-\sigma_e^2$.
The following theorem establishes the asymptotic behavior
of the survival probability of BDREs
including explicit expressions for the limiting constants.
For the rest of this article, we concentrate on the subcritical regime;
the supercritical regime is then subject of the forthcoming
paper~\cite{Hutzenthaler2011ECP}.

\begin{theorem}\label{asymptotic_survival}
  Assume $\al\in\R$ and $\sigma_b,\sigma_e\in(0,\infty)$.
  Let $(Z_t,S_t)_{t\geq0}$ be the strong solution of~\eqref{eq:BDRE}.
  Then we have that
 \begin{alignat}{2}
 \limt\,\P^z\Bigl(Z_t>0\Bigr)
     &=1-\left(1+\frac{\sigma_e^2}{\sigma_b^2}\cdot z\right)^{-\frac{2\al}{\sigma_e^2}}>0
     &\qquad\text{if }&\al>0 \label{thsup}\\
 \limt\sqrt{t}\,\P^z\Bigl(Z_t>0\Bigr)
     &=\frac{\sqrt{2}}{\sqrt{\pi}\sigma_e} \log\Big(1+\frac{\sigma_e^2}{\sigma_b^2} \cdot z\Big)>0
     &\text{if }&\al=0\label{thcrit}\\
 \limt\sqrt{t}^3 e^{\frac{\al^2}{2\sigma_e^2}t}\,\P^z\Bigl(Z_t>0\Bigr)
     &=\frac{8}{\sigma_e^3}\int_0^\infty f(za)\phi_\beta(a)\,da>0
     &\text{if }&\frac{\al}{\sigma_e^2}\in(-1,0)\label{thweak}\\
 \limt\sqrt{t}\, e^{\frac{\sigma_e^2}{2}t}\,\P^z\Bigl(Z_t>0\Bigr)
     &=z\,\frac{\sqrt{2}\sigma_e}{\sqrt{\pi}\sigma_b^2}>0
     &\text{if }&\frac{\al}{\sigma_e^2}=-1\label{thinter}\\
 \limt e^{-\rob{\al+\frac{\sigma_e^2}{2}}t}\,\P^z\Bigl(Z_t>0\Bigr)
     &=z\,2\frac{-\al-\sigma_e^2}{\sigma_b^2}>0
     &\text{if }&\frac{\al}{\sigma_e^2}<-1\label{thstrong}
 \end{alignat}
 for every $z\in(0,\infty)$,
 where $\phi_\beta\colon(0,\infty)\to(0,\infty)$ is defined as
 \begin{equation} \label{defphibeta} \begin{split}
    \phi_\beta(a)&=\int_0^\infty \int_0^\infty \frac{1}{\sqrt{2}\pi}\Gamma\Big(\frac{\beta+2}{2}\Big)
             e^{-a} a^{-\beta/2} u^{(\beta-1)/2}e^{-u} \frac{\sinh(\xi)\cosh(\xi) \xi}{(u+a(\cosh(\xi))^2)^{\frac{\beta+2}{2}}}d\xi\,du
 \end{split}    \end{equation}
 for every $a\in(0,\infty)$.
\end{theorem}
\noindent
The proof is deferred to Section~\ref{sec:asymptotic.of.the.survival.probability}.

Let us compare the convergence rate of the survival probability
with the classical case $\sigma_e=0$ of Feller's branching diffusion.
Pars pro toto we discuss the critical regime $\al=0$.
In that case, the survival probability of Feller's branching diffusion
is of order $O\rob{\tfrac{1}{t}}$ whereas it is of order $O\rob{\tfrac{1}{\sqrt{t}}}$
if $\sigma_e>0$ as $t\to\infty$.
So a branching process in random environment has a higher probability to survive.
The reason for this is that there is a positive probability
of experiencing a long supercritical phase. More precisely, for every $\eps>0$, the
event that the critical Brownian motion $(S_{s})_{s\geq0}$ stays above $\eps$
from time $\eps$
until time $t$ is of order $O\rob{\tfrac{1}{\sqrt{t}}}$ as $t\to\infty$.
On this event the branching process is supercritical and survives with positive
probability. This explains the slower convergence order
$O\rob{\tfrac{1}{\sqrt{t}}}$ as $t\to\infty$.

Note that the expectation $\E^z[Z_t]=z\exp\rob{(\al+\tfrac{\sigma_e^2}{2})t}$, $z\in(0,\infty)$,
changes its qualitative behavior as $t\to\infty$ at $\al=-\tfrac{\sigma_e^2}{2}$.
The phase transition for the survival probability, however, is
at $\al=-\sigma_e^2$. Here is an heuristic.
If the associated Brownian motion (drift $\al<0$) is negative for almost all of the
time, then we expect the BDRE to behave like Feller's branching diffusion.
In that case we expect that $\P^z(Z_t>0)\sim\text{\itshape const}\cdot\E^z[Z_t]=\text{\itshape const}\cdot z\exp\rob{(\al+\tfrac{\sigma_e^2}{2})t}$,
$z\in(0,\infty)$,
as $t\to\infty$.
This gives indeed the exponential decay rate in the strongly subcritical regime.
However, the associated Brownian motion might be positive until time $t>0$.
On this event the BDRE is supercritical and survives with positive probability.
The probability of this event decreases like 
$\exp\rob{-\frac{\al^2}{2\sigma_e^2}t}$ (times polynomial terms)
as $t\to\infty$. This exponential decay rate follows from an application of
the Cameron-Martin-Girsanov theorem
(e.g.~Theorem IV.38.5 in~\cite{RogersWilliams2000b}).
This gives the exponential decay rate in the weakly subcritical regime.
Now the phase transition for the survival probability
occurs when these two exponential decay rates
$\al+\tfrac{\sigma_e^2}{2}$ and $-\frac{\al^2}{2\sigma_e^2}$ coincide,
namely at $\al=-\sigma_e^2$.


For BPREs Afanasyev (1979)\nocite{Afanasyev1979} was the first
to observe different regimes for the survival probability in the
subcritical regime. Independently hereof
Dekking (1987)\nocite{Dekking1987} rediscovered this dichotomy.
For more recent results on the speed of decay of the survival probability,
see
Corollary 1.2 of \cite{AfanasyevEtAl2005AOP} for the critical case,
Corollary 1.2 of \cite{AfanasyevEtAl2011pre} for the weakly subcritical case,
Theorem 1 of \cite{Vatutin2004} for the intermediately subcritical case
and
Theorem 1.1 of \cite{AfanasyevEtAl2005SPA} for the strongly subcritical case.
Its derivation, however, is sometimes involved and,
in general, there are no simple expressions for the limiting constants.
Only
the case of linear-fractional offspring distributions
is known to admit explicit limiting constants, see \cite{Afanasyev1980}.

Next we investigate
the event of survival in more detail
and condition the BDRE on the event of ultimate survival.
The method of conditioning a Markov process to stay nonnegative
has been applied in various situations (e.g.~\cite{BertoinDoney1994,CattiauxEtAl2009,Lambert2007}).
However, we have not found a suitable formulation for the case of
multi-dimensional diffusions.
As such a formulation is of independent interest, we include it in the following lemma.
For this, define a set $\CQ$ of functions as
\begin{equation}
  \CQ:=\left\{q\colon[0,\infty)\to[0,\infty)\Big|\lim_{t\to\infty}\frac{q(t+s)}{q(t)}=1
   \text{ for all }s\in[0,\infty)
   \right\}.
\end{equation}
Note that $q\in\CQ$ if and only if $q\circ \log$ is slowly varying at infinity;
see Galambos and Seneta (1973)\nocite{GalambosSeneta1973} for this notion.
\begin{lemma}  \label{l:conditioning.general}
  Let $d,m\in\N$ and
  let $I\subset\R^d$ and $A\subset I$ be Borel measurable sets.
  In addition let the drift vector $\mu\colon I\to\R^d$
  and the diffusion matrix $\sigma\colon I\to\R^{d\times m}$ be Borel measurable functions.
  Moreover, let $(X_t)_{t\geq0}$ be a Markov process and
  a weak solution of the
  stochastic differential equation
  \begin{equation}  \label{eq:X}
    dX_t=\mu(X_t)\,dt+\sigma(X_t)\,dW_t
  \end{equation}
  for $t\in[0,\infty)$ with initial value $X_0\in I$
  where $(W_t)_{t\geq0}$ is an $m$-dimensional standard Brownian motion.
  Assume that there exist a twice continuously differentiable
  function $\eta\colon I\to[0,\infty)$,
  a function $q\in\CQ$
  and values $\lambda,p\in[0,\infty)$
  with the following properties:
  \begin{itemize}
    \item $\lim_{t\to\infty}q(t)e^{\ld t}\,\P^x\Big(X_t\not\in A\Big)=\eta(x)$
       for all $x\in I$,
    \item  $\sup_{x\in I}\tfrac{1}{1+\|x\|^p}
       \sup_{t\in[1,\infty)}q(t)e^{\ld t}\,\P^x\Big(X_t\not\in A\Big)<\infty$,
    \item $\E^x\left[\|X_t\|^p\right]<\infty$ for all $x\in I$ and $t\in[0,\infty)$.
  \end{itemize}
  Define $\bar{I}:=\{x\in I\colon \eta(x)>0\}$.
  Then there exists a process $(\Xb_t)_{t\geq0}$ with state space
  $\bar{I}$ such that
  \begin{equation}  \label{eq:conditioning.general}
    \P^x\left(\left(X_s\right)_{s\in[0,t]}\in\bullet\,\Big|X_T\not\in A\right)
    \xrightarrow[T\ra \infty]{\text{w}}
    \P^x\left(\left(\Xb_s\right)_{s\in[0,t]}\in\bullet\right)
  \end{equation}
  for all $x\in \bar{I}$ and all $t\in[0,\infty)$,
  such that  $(\Xb_t)_{t\geq0}$
  is a weak solution of the SDE
  \begin{equation}  \label{eq:Xbar}
    d\Xb_t=
        \left(
             \sigma\sigma^{t}\frac{\nabla^{t} \eta}{\eta}
        \right)(\Xb_t)\,dt
          +\mu(\Xb_t)\,dt+\sigma(\Xb_t)\,dW_t
  \end{equation}
  for $t\in[0,\infty)$
  where $\nabla=\left(\tfrac{\del}{\del x_1},\ldots,\tfrac{\del}{\del x_d}\right)$
  and
  such that
  \begin{equation}  \label{eq:semigroup.Xbar}
    \E^x\left[g(\Xb_t)\right]=\frac{\E^x\left[\eta(X_t)g(X_t)\right]}{e^{-\ld t}\eta(x)}
  \end{equation}
  for all $x\in \bar{I}$, $t\in[0,\infty)$ and all Borel measurable functions $g\colon I\to[0,\infty)$.
\end{lemma}
\noindent
The proof is deferred to Section~\ref{sec:conditioning.general}.

We will apply Lemma~\ref{l:conditioning.general} to our bivariate process
$(Z_t,S_t)_{t\geq0}$ and to the set $A=\{0\}\times\R$.
Lemma~\ref{l:conditioning.general} shows that the limiting constants
of Theorem~\ref{asymptotic_survival} play an important role for conditioning
on ultimate survival.
Theorem~\ref{asymptotic_survival} shows that $\eta(z,s)=\vartheta(z)$,
$(z,s)\in[0,\infty)\times\R$, where
the function $\vartheta\colon[0,\infty)\to[0,\infty)$ is defined through
\begin{equation}\label{eq:vartheta}
  \vartheta(z)=
  \begin{cases}
     1-\Big(1+\frac{\sigma_e^2}{\sigma_b^2}\cdot z\Big)^{-\frac{2\alpha}{\sigma_e^2}}  &\text{if }\al>0\\
    \frac{\sqrt{2}}{\sqrt{\pi}\sigma_e} \log\Big(1+\frac{\sigma_e^2}{\sigma_b^2} \cdot z\Big)          & \text{if }\al=0\\
    \frac{8}{\sigma_e^3}\int_0^\infty f(za)\phi_\beta(a)\,da     & \text{if }\tfrac{\al}{\sigma_e^2}\in(-1,0)\\
    z\,\frac{\sqrt{2}\sigma_e}{\sqrt{\pi}\sigma_b^2}>0   & \text{if }\tfrac{\al}{\sigma_e^2}=-1\\
    z\, 2\Big(\frac{-\alpha-\sigma_e^2}{\sigma_b^2}\Big)
                                                                 & \text{if }\tfrac{\al}{\sigma_e^2}<-1
  \end{cases}
\end{equation}
for every $z\in[0,\infty)$
where $\phi_\beta$
is defined in Theorem~\ref{asymptotic_survival}.
Moreover, define $\ld:=0$ for $\al\geq0$,
$\ld:=\tfrac{\al^2}{2\sigma_e^2}$ for $\al\in(-\sigma_e^2,0)$
and $\ld:=-(\al+\tfrac{\sigma_e^2}{2})$ for $\al\leq -\sigma_e^2$.
The following theorem characterizes the BDRE conditioned
to never go extinct.
\begin{theorem}  \label{thm:conditioning.BDRE}
  Assume $\al\in\R$ and $\sigma_b,\sigma_e\in(0,\infty)$.
  Let $(Z_t,S_t)_{t\geq0}$ be the strong solution of~\eqref{eq:BDRE}.
  Let $(\Zb_t,\Sb_t)_{t\geq0}$ denote the process $(Z_t,S_t)_{t\geq 0}$
  conditioned to never go extinct.
  Then this process is a weak solution of the SDEs
  \begin{equation}  \begin{split}  \label{eq:SDE.conditioned}
    d\Zb_t&=\left(\sigma_b^2\frac{\vartheta^{'}(\Zb_t)}{\vartheta(\Zb_t)}\Zb_t+\frac{1}{2}\sigma_e^2\Zb_t
           \right)\,dt
          +\Zb_td\Sb_t+\sqrt{\sigma_b^2\Zb_t}dW_t^{(b)}\\
    d\Sb_t&=\left(\al+\sigma_e^2\frac{\vartheta^{'}(\Zb_t)}{\vartheta(\Zb_t)}\Zb_t\right)\,dt
            +\sigma_e dW_t^{(e)}
  \end{split}     \end{equation}
  for $t\in[0,\infty)$ and satisfies
  \begin{equation}  \label{eq:semigroup.Zb}
    \E^{(z,s)}\left[g\left(\Zb_t,\Sb_t\right)\right]
    =\frac{e^{\ld t}}{\vartheta(z)}\E^{(z,s)}\left[\vartheta(Z_t)g(Z_t,S_t)\right]
  \end{equation}
  for all $(z,s)\in(0,\infty)\times\R$, $t\in[0,\infty)$ and all Borel measurable
  functions $g\colon[0,\infty)\times\R\to[0,\infty)$.
  If $\al\in(-\sigma_e^2,\infty)$, then the function
  $(0,\infty)\ni z\mapsto \sigma_e^2 z\vartheta^{'}(z)/\vartheta(z)\in\R$
  is strictly monotonic decreasing
  and
  satisfies $\lim_{z\to 0}\sigma_e^2 z\vartheta^{'}(z)/\vartheta(z)=\sigma_e^2$
  and $\lim_{z\to\infty}\sigma_e^2 z\vartheta^{'}(z)/\vartheta(z)=\max(-\al,0)$.
  If $\al\in(-\infty,-\sigma_e^2]$, then
  $\sigma_e^2 z\vartheta^{'}(z)/\vartheta(z)=\sigma_e^2$
  for all $z\in(0,\infty)$.
  If $\al>-\sigma_e^2$,
  then $\lim_{t\to\infty}\Zb_t=\infty$ in distribution.
  If $\al<-\sigma_e^2$, then
  \begin{equation}  \label{eq:invariant}
    \P^z\left(\Zb_t\in dy\right)\varwlimt
    c\, y\left(\sigma_b^2+\sigma_e^2 y\right)^{\frac{2\al}{\sigma_e^2}}\,dy
  \end{equation}
  (weak convergence of measures on $(0,\infty)$)
  for every $z\in(0,\infty)$ where
  $c=c(\al,\sigma_b^2,\sigma_e^2)\in(0,\infty)$ is a normalizing constant
  such that the right-hand side is a probability distribution.
\end{theorem}
\noindent
The proof is deferred to Section~\ref{sec:The BDRE conditioned to go never extinct}.

Theorem~\ref{thm:conditioning.BDRE} exhibits a difference
in the survival opportunities
between the weakly subcritical and stronlgy subcritical regimes.
In the strongly subcritical regime, $\vartheta$ is a linear function
and the SDEs of the
conditioned process $(\Zb_t,\Sb_t)_{t\geq0}$
simplify to
  \begin{equation}  \begin{split}  \label{eq:Zb.linear}
    d\Zb_t&=\left(\sigma_b^2+\frac{1}{2}\sigma_e^2\Zb_t
           \right)\,dt
          +\Zb_td\Sb_t+\sqrt{\sigma_b^2\Zb_t}dW_t^{(b)}\\
    d\Sb_t&=\left(\al+\sigma_e^2\right)\,dt
            +\sigma_e dW_t^{(e)}
  \end{split}     \end{equation}
for $t\in[0,\infty)$.
The drift of the associated Brownian motion $(\Sb_t)_{t\geq0}$ is increased by $\sigma_e^2$, but is still negative.
Thus the conditioned process survives solely due to the
immigration term $\sigma_b^2\,dt$. 
In the weakly subcritical regime,
the drift of the associated Brownian motion $(\Sb_t)_{t\geq0}$ is strictly positive.
Thus the conditioned process in the weakly subcritical regime survives due to a supercritical
environment with positive probability.
The immigration term $\sigma_b^2\tfrac{\vartheta^{'}(\Zb_t)}{\vartheta(\Zb_t)}\Zb_t\,dt$
is not needed for this. The effect of this immigration term is to ensure
survival with full probability.
Another observation in the weakly subcritical regime is
that the environment in the conditioned process
depends on the population size.
The reason for this is that the survival probability of a supercritical BDRE depends
on the initial mass.
In addition
it is intuitive that
the environment in the conditioned process
needs to be less beneficial if the population size is large.
Formally this means that the additional drift term
 $\sigma_e^2 z \vartheta^{'}(z)/\vartheta(z)$ is decreasing in $z\in(0,\infty)$.
More precisely, this function decreases
from $\sigma_e^2$ to $\max(-\al,0)$
as the population size increases.
Moreover, Theorem~\ref{thm:conditioning.BDRE} provides an explicit
quantification of the dependence of the additional drift term in the
conditioned process on the population size.

In the strongly subcritical and in the intermediately subcritical regimes,
conditioning on ultimate survival affects each individual in the same way
so that the conditioned process is again a BDRE except for an additional
immigration term.
In the case of a constant environment,
branching processes with immigration may be represented as a branching process
with an additional immortal individual.
The trajectory of the immortal individual is referred to as spine or backbone.
This backbone construction goes back to Kallenberg (1977)\nocite{Kallenberg1977}
for branching processes in discrete time and has later been established
e.g.~for branching processes in continuous time
(Gorostiza and Wakolbinger 1991\nocite{GorostizaWakolbinger1991}),
for the Dawson-Watanabe superprocess (Evans 1993\nocite{Evans1993}),
for the infinite-variance $(1+\beta)$-superprocess
(Etheridge and Williams 2003\nocite{EtheridgeWilliams2003})
or for general continuous-state branching processes (Lambert 2007\nocite{Lambert2007}).
In all of these backbone constructions, families evolve independently of each other.
The next theorem establishes the backbone construction for BDREs conditioned on ultimate survival
in the strongly subcritical and in the intermediately subcritical regimes.
Here the families are correlated through the environment.
Due to Proposition~\ref{p:time.change}, however,
this correlation is rather explicit.

We understand a family to be a single ancestor together with the progeny
of that individual.
The total mass hereof as a function of time
is an excursion from $0$
as a single individual has mass $0$ in the diffusion approximation.
We denote the space of continuous excursions from $0$ as
\begin{equation}  \label{eq:excursion.space}
  U:=\left\{\chi\in\C\left((-\infty,\infty),[0,\infty)\right)\colon
    T_0(\chi)\in(0,\infty],\,\chi_t=0\;\,\forall\,
    t\in(-\infty,0]\cup[T_0(\chi),\infty)\right\}
\end{equation}
where $T_0(\chi):=\inf\{t>0\colon\chi_t=0\}\in[0,\infty]$ is the 
first hitting time of $0$ ($\inf\emptyset:=\infty$).
Let $(F_t)_{t\geq0}$ be the strong solution of the SDE
\begin{equation}
  dF_t= \sqrt{F_t}dW_t
\end{equation}
for $t\in[0,\infty)$.
The law of families of the process $(F_t)_{t\geq0}$ is the excursion measure $Q$
which is a $\sigma$-finite measure on the excursion space $U$
and which is uniquely determined by
\begin{equation}   \label{eq:excursion.measure}
  \int g\left(\chi\right)\,Q(d\chi)
  =\lim_{\delta\rightarrow0}\frac{1}{\delta}\,\E^{\delta}\Big[g\left((F_t)_{t\geq0}\right)\Big]
\end{equation}
for all bounded, continuous functions $g\colon\C\left([0,\infty),[0,\infty)\right)\to\R$ depending on a finite time interval and such that there is an $\eps>0$ such that
$g(\chi)=0$
for all $\chi\in\C\left([0,\infty),[0,\infty)\right)$ with
$\sup_{t\geq0}\chi_t\leq\eps$.
Such a measure $Q_F$ exists according to Theorem 1 in~\cite{Hutzenthaler2009EJP}.
\begin{theorem}  \label{thm:family.decomposition}
  Assume $\sigma_b,\sigma_e\in(0,\infty)$ and $\al\in(-\infty,-\sigma_e^2]$.
  Let $(W_t^{(e)})_{t\geq0}$ be a standard Brownian motion.
  Define
  \begin{itemize}
    \item   $\St_t:=(\al+\sigma_e^2) t+\sigma_e W_t^{(e)}$ for $t\in[0,\infty)$,\vspace{-3mm}
    \item   $(\tilde{\tau}(t))_{t\geq0}$
           through $\tilde{\tau}(t):=\int_0^t e^{-\St_u}\sigma_b^2\,du$
           for $t\in[0,\infty)$,\vspace{-3mm}
    \item a Poisson point process $\CP$ on $[0,\infty)\times U$ with intensity measure $dy\times Q$ and\vspace{-3mm}
    \item a Poisson point process $\tilde{\CP}$ on $[0,\infty)\times U$
           with intensity measure $dt\times Q$.\vspace{-1mm}
  \end{itemize}
  Assume the ingredients $(W_t^{(e)})_{t\geq0}$, $\CP$ and $\tilde{\CP}$ to be independent.
  Let $(\Zb_t,\Sb_t)_{t\geq0}$ denote the BDRE $(Z_t,S_t)_{t\geq0}$ started in $Z_0=z\in(0,\infty)$
  and conditioned
  to never go extinct
  and
  define a process $(\Zt_t)_{t\geq0}$ through $\Zt_0:=z$ and through
  \begin{equation}  \label{eq:family.decomposition}
    \Zt_t:=\sum_{(y,\chi)\in\CP}\1_{y\leq z}\,\chi_{\tilde{\tau}(t)}e^{\St_t}
    +\sum_{(u,\chi)\in\tilde{\CP}}\chi_{\tilde{\tau}(t)-u}e^{\St_t}
  \end{equation}
  for $t\in(0,\infty)$.
  Then $(\Zb_t,\Sb_t)_{t\geq0}$ and $(\Zt_t,\St_t)_{t\geq0}$ are equal in distribution.
\end{theorem}
\noindent
The proof is deferred to Section~\ref{sec:family.decomposition}.

The excursions $(y,\chi)\in\CP$ are the families whose ancestor lived before time $0$.
Due to conditioning on ultimate survival, there is an immortal individual.
Offspring of this individual are the ancestors of families
$(s,\chi)\in\tilde{\CP}$.
Conditioned on the environment, all of these families evolve
independently of each other.
Note that the environment appears in~\eqref{eq:family.decomposition}
only through the time-change and through the reweighting of the
critical excursion paths.
In the critical and weakly subcritical regimes, the conditioned process
$(\Zb_t,\Sb_t)_{t\geq0}$
is not a BDRE with immigration.
A representation with independent families is therefore not possible.
In view of Theorem~\ref{thm:family.decomposition},
the SDEs~\eqref{eq:SDE.conditioned} can still be interpreted as follows:
Birth events of the immortal individual are accepted only with
probability $z\vartheta^{'}(z)/\vartheta(z)\in(0,1)$ if the current
population size is $z\in(0,\infty)$.
Moreover the additional drift of $(\Sb_t)_{t\geq0}$
is not $\sigma_e^2$ as in the strongly subcritical regime
but $\sigma_e^2 z\vartheta^{'}(z)/\vartheta(z)\in(0,\sigma_e^2)$ if the current
population size is $z\in(0,\infty)$.

\section{Diffusion approximation}
\label{sec:diffusion.approximation}
\begin{proof}[{Proof of Proposition~\ref{p:diffusion.approximation}}]
 We derive Proposition~\ref{p:diffusion.approximation}
 from
 Corollary 2.18 of Kurtz (1978)\nocite{Kurtz1978}.
 To check the assumptions hereof we need more notation.
 Define
 \begin{equation}
   \alpha_i^{(n)}
   :=\sum_{k=0}^\infty \robb{\frac{k}{m\rob{Q_i^{(n)}}}-1}^2
                    Q_i^{(n)}\rob{k}
 \end{equation}
 for all $i\in\N,n\in\N$
 and note that $\alpha_i^{(n)}$, $i\in\N$, are independent
 and identically distributed
 for every $n\in\N$.
 According to assumption~\eqref{eq:assumption.var.b},
 the expectation of $\alpha_i^{(n)}$ converges to $\sigma_b^2$
 as $n\to\infty$ for every $i\in\N$.
 Therefore the law of large numbers for triangular independent sequences implies that
 \begin{equation}
   A_n(t):=\frac{1}{n}\sum_{i=1}^{\floor{nt}}\alpha_i^{(n)}
   \to t\, \sigma_b^2\qqasn\quad
 \end{equation}
 almost surely for every $t\in[0,\infty)$.
 The rescaled associated random walk converges (see \cite{RavkauskasSuquet2003}, Theorem 3) to a Brownian motion
 $(S_t)_{t\geq0}$, that is,
 \begin{equation}
   \robb{\frac{S_{tn}^{(n)}}{\sqrt{n}}}_{t\geq0}
   \varwlim \rob{S_t}_{t\geq0} \qqasn.
 \end{equation}
 The Brownian motion $(S_t)_{t\geq0}$ has drift $\alpha$
 due to assumption~\eqref{eq:assumption.mean} and
 due to $\log(x)\approx x-1$ for all $x$ in a neighbourhood of $1$.
 Furthermore, $(S_t)_{t\geq0}$ has infinitesimal variance $\sigma_e^2$
 due to assumption~\eqref{eq:assumption.var.e}.
 Moreover,
 Corollary 2.18 of Kurtz (1978)\nocite{Kurtz1978}
 requires a third moment condition which follows
 from~\eqref{eq:assumption.third.moment} and from
 \begin{equation}  
   \E\sqbb{\frac{1}{n^{3/2}}
    \sum_{i=0}^{\floor{nt}-1}
    \sum_{k=0}^\infty\absbb{\frac{k}{m\rob{Q_i^{(n)}}}-1}^3
         Q_i^{(n)}(k)
   }\leq
   \frac{t}{\sqrt{n}}\sup_{\nb\in\N}
   \E\sqbb{
    \sum_{k=0}^\infty\absbb{\frac{k}{m\rob{Q_0^{(\nb)}}}-1}^3
         \cdot Q_0^{(\nb)}(k)
   }\stackrel{n\rightarrow\infty}{\longrightarrow}0.\nonumber
   \end{equation}
 Having checked all assumptions,
 Proposition~\ref{p:diffusion.approximation}
 follows from
 Corollary 2.18 of Kurtz (1978)\nocite{Kurtz1978}.
\end{proof}
\section{The Laplace transform and the extinction probability}
\label{sec:Laplace.transform}

\begin{proof}[{Proof of Corollary~\ref{c:Laplace.transform}}]
 It suffices to prove~\eqref{eq:Laplace.transform}
 for the version~\eqref{eq:time.change} of the BDRE
 due to Proposition~\ref{p:time.change}.
 The Laplace transform of Feller's branching diffusion
 $(F_t)_{t\geq0}$ satisfies
 \begin{equation}  \label{eq:Laplace.transform.Feller}
   \E^z\sqB{\exp\rob{-\ld F_t}}
   =\exp\roB{-\frac{z}{\frac{1}{2}t
                       +\frac{1}{\ld}}}
   \quad \text{for }t,z,\ld\in[0,\infty),
 \end{equation}
 (e.g., Example 26.11 of~\cite{Klenke2008}).
 Thus we get for the Laplace transform of $F_{\tau(t)}e^{S_t}$ that
 \begin{equation}  \begin{split}
   \E^{(z,0)}\sqB{\exp\rob{-\ld F_{\tau(t)}e^{S_t}}|\rob{S_s}_{s\leq t}}
   &=\exp\roB{-\frac{z}{\frac{1}{2}\tau(t)
                       +\frac{1}{\ld e^{S_t}}}}
   \\
   &=
   \exp\bigg(-\frac{z}{\int_0^t \frac{\sigma_b^2}{2}
        \exp\big(-S_s \big) ds
        + \frac{1}{\lambda} \exp(-S_t)
        }
        \bigg)
 \end{split}     \end{equation}
 for all $t,z,\ld\in[0,\infty)$ almost surely.
 This completes the proof.
\end{proof}
%
%
\begin{proof}[{Proof of Corollary~\ref{c:survival_prob}}]
 Fix $t\in[0,\infty)$ and $z\in[0,\infty)$.
 Letting $\ld\to\infty$ in formula~\eqref{eq:Laplace.transform} for
 the Laplace transform and applying the dominated convergence theorem yields that
 \begin{equation}  \begin{split}  \label{eq:proof.survival.prob}
   \lefteqn{
   \P^{(z,0)}\Bigl(Z_t>0\,\big|\,(S_s)_{s\leq t}\Bigr)
   =\lim_{\ld\to\infty}\E^{(z,0)}\sqB{1-\exp\rob{-\ld Z_t}\big|\,(S_s)_{s\leq t}}
   }\\
   &=\lim_{\ld\to\infty}\sqbb{1-\exp\bigg(-\frac{z}{\int_0^t \frac{\sigma_b^2}{2}
        \exp\big(-S_s \big) ds
        + \frac{1}{\lambda} \exp(-S_t)
        }
        \bigg)}\\
   &=1-\exp\biggl(-\frac{z}{\int_0^t\frac{\sigma_b^2}{2}
                 \exp\bigl(-S_s\bigr)ds}\biggr)
 \end{split}     \end{equation}
 almost surely.
 This proves~\eqref{survprob}.
 Now assume that $\sigma_e>0$ and recall that $\beta=-\tfrac{2\al}{\sigma_e^2}$.
 Note that
 \begin{equation} \label{eq:scaling_BM}
   \Big(-\alpha s - \sigma_e W^{(e)}_s\Big)_{0\leq s\leq t}\stackrel{d}{=}
   \bigg(-2\frac{2\al}{\sigma_e^2}\frac{\sigma_e^2}{4}s+2W^{(e)}_{\frac{\sigma_e^2}{4}s}\bigg)_{0\leq s\leq t}.
 \end{equation}
 Inserting $S_s=\alpha s+\sigma_e W^{(e)}_s$, $s\leq t$, into~\eqref{eq:proof.survival.prob},
 applying~\eqref{eq:scaling_BM} and 
 the time substitution $u:=\tfrac{\sigma_e^2}{4}s$,
 we get that
 \begin{equation}  \begin{split}
   \P^z\Bigl(Z_t>0\,\big|\,(S_s)_{s\leq t}\Bigr)
   &=1-\exp\biggl(-\frac{z}{\int_0^t\frac{\sigma_b^2}{2}
      \exp\bigl(-\al s-\sigma_e W^{(e)}_s\bigr)ds}\biggr)\\
   &\eqd 1-\exp\biggl(-\frac{z}{\int_0^t\frac{\sigma_b^2}{2}
      \exp\Bigl(-2\frac{2\al}{\sigma_e^2}\frac{\sigma_e^2}{4}s+2W^{(e)}_{\frac{\sigma_e^2}{4}s}\Bigr)ds}\biggr)\\
   &= 1-\exp\biggl(-\frac{z}{\int_0^{\sigma_e^2 t/4}\frac{\sigma_b^2}{2}\frac{4}{\sigma_e^2}
      \exp\bigl(2\beta u+2W^{(e)}_{u}\bigr)du}\biggr)\\
   &= 1-\exp\biggl(-\frac{z\sigma_e^2}{2\sigma_b^2 A_{\sigma_e^2t/4}^{(\beta)}}\biggr)
    = f\biggl(\frac{z}{2 A_{\sigma_e^2t/4}^{(\beta)}}\biggr)
 \end{split}     \end{equation}
 almost surely.
 Taking expectations, we arrive at
\begin{equation} \label{eq:proof.survneu} 
\P^z\big(Z_t>0\big)
 =\E\bigg[f\bigg(\frac{z}{2A^{(\beta)}_{t\sigma_e^2/4}}\bigg)\bigg]
 =\int_0^{\infty}f(za) p_{t\sigma_e^2/4,\beta}(a)da.
\end{equation}
The last step is equation (2.5) in Matsumoto and Yor (2003)\nocite{MatsumotoYor2003}
which requires $\beta>-1$.
\end{proof}
\section{Asymptotic behavior of the survival probability}
\label{sec:asymptotic.of.the.survival.probability}
Throughout this section, let $(Z_t,S_t)_{t\geq0}$
be the strong solution of~\eqref{eq:BDRE} and $f$ be defined as in (\ref{deff}).
\subsection{General results}
%
Recall $A_t^{(\gamma)}$, $t\in[0,\infty)$, $\gamma\in\R$, from~\eqref{eq:At}.
First we provide sufficient conditions under which
$\E\big[z/A_t^{(\gamma)}]$ and  $\E\big[(1-\exp(-z/A_t^{(\gamma)}))\big]$, $z\geq 0$,
have the same asymptotics as $t\to\infty$ for $\gamma\in\R$.
We will use this for the intermediately and strongly subcritical regimes.
\begin{lemma}\label{hilf}
Let $(Y_t)_{t\geq1}$ be a family of non-negative random variables.
Assume that there exist a function
$c\colon[1,\infty)\to[0,\infty)$ and a constant $a\in[0,\infty)$ such that
 $\lim_{t\rightarrow\infty} c_t\E\big[Y_t\big]= a$ and $\limsup_{t\rightarrow\infty} c_t\E\big[Y^2_t\big]=0$.
Then we have that
 \begin{align} \label{eq:hilf}
  \lim_{t\rightarrow\infty} c_t\E\Big[1-\exp\big(-\ld Y_t\big)\Big]=\ld a
 \end{align}
 for every $\ld \in[0,\infty)$.
\end{lemma}
\begin{proof}
 Let $\ld \in[0,\infty)$ be fixed.
The upper bound follows from $1-e^{-\ld x}\leq \ld x$, $x\geq 0$ and from $\lim_{t\rightarrow\infty} c_t\E\big[Y_t\big]= a$. 
The lower bound results from $1-e^{-\lambda x}\geq \lambda x-\frac{\lambda^2x^2}{2}$ for every $x\geq0$. Applying this, we get that
 \begin{align}
    c_t \E\Big[1-\exp(-\ld Y_t)\Big] 
    &\geq c_t\E\Big[\ld Y_t-\frac{\ld^2Y_t^2}{2}\Big]
    = \ld  c_t \E[Y_t] -c_t\frac{\ld^2\E\big[Y_t^2\big]}{2} \label{2506}
 \end{align}
 for every $t\geq0$.
The assumptions of the lemma then yield that
 \begin{align}
   \liminf_{t\rightarrow\infty}c_t\E\Big[1-\exp\big(-\ld  Y_t\big)\Big]
   \geq \liminf_{t\rightarrow\infty} \Big(\ld  c_t \E[Y_t] -\frac{\ld^2c_t\E\big[Y_t^2\big]}{2}\Big) = \ld  a ,
 \end{align}
 which is the lower bound in~\eqref{eq:hilf}.
\end{proof}

The Cameron-Martin-Girsanov theorem
(e.g.,~Theorem IV.38.5 in~\cite{RogersWilliams2000b})
for a standard Brownian motion $(W_s)_{s\geq 0}$
asserts that
\begin{align}
\E\Big[h\big((W_s+\theta s)_{0\leq s\leq t}\big)\Big]
= \E\Big[\exp(\theta W_t-\theta^2 t/2) h\big( (W_s)_{0\leq s\leq t}\big)\Big] \label{cameron}
\end{align}
for every $\theta\in\R$, every measurable function
$h\colon \C\left([0,t],\R\right)\to[0,\infty)$ and every $t\in[0,\infty)$.
The next lemma will allow us to deduce the intermediately and the strongly subcritical case
from the critical and  supercritical case, respectively,
by changing the drift through (\ref{cameron}).
\begin{lemma}\label{lem_hilf_intstrong}
Let $\gamma\in\R$ and let $(A^{(\gamma)}_t)_{t\geq 0}$ be defined as in (\ref{eq:At}). Then we have that
\begin{align}
\E\Big[\frac{1}{2A_t^{(\gamma)}}\Big]&= e^{-(2\gamma-2)t}\E\Big[\frac{1}{2A_t^{(-(\gamma-2))}}\Big]\label{ew0608}
\\
\E\Big[\frac{1}{(2A_t^{(\gamma)})^2}\Big]&\leq e^{-(2\gamma-2)t}\E\Big[\frac{1}{2A_{t/2}^{(\gamma-2)}}\Big]\cdot\E\Big[\frac{1}{2A_{t/2}^{(-(\gamma-2))}}\Big],
\end{align}
for every $t\in(0,\infty)$.
\end{lemma}
\begin{proof}
Let $t\in(0,\infty)$ and $\gamma\in\R$ be fixed.
Applying (\ref{cameron}) with $\theta=2$, we obtain that
\begin{align}
&\E\Big[\frac{1}{2A_t^{(\gamma)}}\Big]= \E\Big[\frac{1}{2\int_0^t \exp\big(2(\gamma s+W^{(e)}_s)\big)ds}\Big]=  \E\Big[\frac{\exp(2W^{(e)}_t-2^2 t/2)}{2\int_0^t \exp\big(2((\gamma-2) s+W^{(e)}_s)\big)ds}\Big]\nonumber\\
&=  e^{-2t}\E\Big[\frac{1}{2\int_0^t \exp\big(2((\gamma-2) s+W^{(e)}_s-W^{(e)}_t)\big)ds}\Big]=  e^{-2t}\E\Big[\frac{1}{2\int_0^t \exp\big(2((\gamma-2) s+W^{(e)}_{t-s})\big)ds}\Big]\nonumber\\
&=  e^{-(2\gamma-2)t}\E\Big[\frac{1}{2\int_0^t \exp\big(2(-(\gamma-2)u +W^{(e)}_u)\big)du}\Big]=e^{-(2\gamma-2)t}\E\Big[\frac{1}{2A_t^{(-(\gamma-2))}}\Big],\label{haupt}\nonumber
\end{align}
where we used the substitution $u:=t-s$.
This is the first claim of the lemma.

For the second moment of $1/(2A_t^{(\gamma)})$, we use analogous arguments to obtain that
\begin{align}
\E\Big[\frac{1}{(2A_t^{(\gamma)})^2}\Big]= \E\Big[\frac{1}{\big(2\int_0^t \exp\big(2(\gamma s+W^{(e)}_s)\big)ds\big)^2}\Big]
=  \E\Big[\frac{\exp(2W^{(e)}_t-2^2 t/2)}{\big(2\int_0^t \exp\big(2((\gamma-2) s+W^{(e)}_s)\big)ds\big)^2}\Big].\nonumber
\end{align}
A rough estimate allows us to split the integral into two independent parts:
\begin{equation}  \begin{split}
\E\Big[&\frac{1}{(2A_t^{(\gamma)})^2}\Big]= e^{-2t}\E\Big[\frac{\exp(2W_t^{(e)})}{\big(2\int_0^t \exp\big(2((\gamma-2) s+W^{(e)}_{s})\big)ds\big)^2}\Big] \\
&\leq e^{-2t} \E\bigg[\frac{1}{2\int_0^{t/2} \exp\big(2((\gamma-2)s +W^{(e)}_s)\big)ds}\cdot\frac{\exp(2W_t^{(e)})}{2\int_{t/2}^t \exp\big(2((\gamma-2)s +W^{(e)}_s)\big)ds}\bigg]\\
&= e^{-2t}\E\Big[\frac{1}{2A_{t/2}^{(\gamma-2)}}\Big]\cdot \mathbb{E}\bigg[\frac{1}{2\int_{t/2}^t \exp\big(2((\gamma-2)s +W^{(e)}_s-W_t^{(e)}))ds}\bigg]\\
\end{split}     \end{equation}
Substituting $u:=t-s$ yields that
\begin{align}
\begin{split}
\E\Big[&\frac{1}{(2A_t^{(\gamma)})^2}\Big]\leq  e^{-2t}\E\Big[\frac{1}{2A_{t/2}^{(\gamma-2)}}\Big]\cdot \mathbb{E}\bigg[\frac{\exp(-2(\gamma-2)t)}{2\int_{t/2}^t 
\exp\big(2((\gamma-2)(s-t) +W^{(e)}_{t-s})\big)ds}\bigg]\\
&=e^{-2t-2(\gamma-2)t}\E\Big[\frac{1}{2A_{t/2}^{(\gamma-2)}}\Big]\cdot \mathbb{E}\bigg[\frac{1}{2\int_{0}^{t/2} \exp\big(2(-(\gamma-2)u +W^{(e)}_u)\big)ds}\bigg] \\
&=e^{-(2\gamma-2)t} \E\Big[\frac{1}{2A_{t/2}^{(\gamma-2)}}\Big]\cdot\E\Big[\frac{1}{2A_{t/2}^{(-(\gamma-2))}}\Big],
\end{split}
\end{align}
which is the second claim of the lemma.
\end{proof}

\subsection{The supercritical regime}
In this subsection, we will prove the result of Theorem \ref{asymptotic_survival}
in the case of $(S_t)_{t\geq0}$ having positive drift $\alpha>0$.
Here we will prove for $\alpha>0$ that, starting from $z>0$, 
the probability of survival is strictly positive 
in the limit as $t\to\infty$ and for every $z\in[0,\infty)$ is given by
 \begin{equation}  \begin{split} \label{eq:survival_asymptotic_alsup}
   \lim_{t\rightarrow\infty} \P^z(Z_t>0) 
   =\E\Big[1-\exp\Big(-\frac{\sigma_e^2}{\sigma_b^2}\, z\, G_{\frac{2\alpha}{\sigma_e^2}}\Big)\Big]
   =1-\left(1+\frac{\sigma_e^2}{\sigma_b^2}\cdot z\right)^{-\frac{2\al}{\sigma_e^2}},
 \end{split}     \end{equation}
where
\begin{equation}  \label{eq:gamma.distribution}
  \P(G_\nu\in dx)= \frac{1}{\Gamma(\nu)}x^{\nu-1}e^{-x}\,dx,\quad x\in(0,\infty).
\end{equation}
The last equality in~\eqref{eq:survival_asymptotic_alsup}
follows from the explicit formula for the Laplace transform of the gamma-distribution 
(e.g.~\cite{Severini2005}
), that is, for every $\ld\geq 0$ and $\nu\in(0,\infty)$,
$\mathbb{E}[\exp(-\ld G_{\nu})] = \frac{1}{(1+\ld)^\nu}$.
We use the following result of
Dufresne (1990)\nocite{Dufresne1990} (see also~\cite{Yor1992JAP}).
\begin{lemma}\label{lem_dufresne}
Let $(B_t)_{t\geq0}$ be a standard Brownian motion.
For all $a\neq 0$ and $b>0$,
 \begin{equation}
   \Big(\int_0^\infty \exp(a B_s-bs)\,ds\Big)^{-1}
   \overset{\rm d}{=} \frac{a^2}{2}G_{\frac{2b}{a^2}}
 \end{equation}
 where $G_\nu$
 has distribution~\eqref{eq:gamma.distribution}.
\end{lemma}
\begin{proof}[{Proof of the supercritical case of Theorem \ref{asymptotic_survival}}]
 Let $z\geq0$ be fixed.
 As $1-e^{-x}\leq 1$, $x\geq 0$,
 we may apply the dominated convergence theorem.
 Using (\ref{survprob}) in Corollary \ref{c:survival_prob} and continuity of $f$, we get that
 \begin{equation}  \begin{split}
   \lim_{t\rightarrow\infty} \P^z(Z_t>0)
   &=\lim_{t\rightarrow\infty}
     \E\bigg[1-\exp\bigg(-\frac{2 z}{\sigma_b^2\int_0^t\exp\big(-\alpha s-\sigma_e W^{(e)}_s\big) ds}\bigg)\bigg]\\
   &=\E\Big[1-\exp\Big(-\frac{\sigma_e^2}{\sigma_b^2}\, z\, G_{\frac{2\alpha}{\sigma_e^2}}\Big)\Big].
 \end{split}     \end{equation}
\end{proof}

\subsection{The critical regime}
Next we study the case $\alpha=0$
and prove (\ref{thcrit}).
For simplicity, denote $A_t:=A_t^{(0)}$ for $t\geq0$.
The next lemma yields the convergence.
\begin{lemma}\label{l:convergencecrit}
Let $g\colon[0,\infty)\rightarrow [0,\infty)$ be a Borel measurable function.
Assume for some constants $c,p\in(0,\infty)$
that $g(x)\leq c\, x^p$ for every $x\geq 0$. Then we get that 
\begin{align}
\lim_{t\rightarrow\infty} \sqrt{t}\cdot \E\Big[g\big(\frac{1}{2A_{t}}\big)\Big] = \int_0^\infty g(a) \cdot \frac{1}{\sqrt{2\pi}}\frac{e^{-a}}{a} da<\infty.
\end{align}
\end{lemma}
\begin{proof}
Let $g$ be a function fulfilling the conditions of the lemma. Instead of (\ref{densitity_general}),
we will use a simpler expression for the density of $A_t$ for $t\in(0,\infty)$.
According to Theorem 4.1 of Dufresne (2001)\nocite{Dufresne2001},
the density function of $1/(2A_t)$ is given by
\begin{align}   \label{eq:density_area}
 \P\Big(\frac{1}{2A_t}\in da\Big)
&=\frac{\sqrt{2}e^{\frac{\pi^2}{8t}}}{\sqrt{\pi^2 t}}\frac{1}{\sqrt{a}}
 \int_0^\infty\exp\Big(-a\big(\cosh(y)\big)^2-\frac{y^2}{2t}\Big)
   \cosh(y)\cos\Big(\frac{\pi y}{2t}\Big)\,dy \,da \nonumber \\
 & =:p_t(a)\,da
\end{align}
on $(0,\infty)$ for every $t\in(0,\infty)$.
As for every $x\geq 0$,  $\frac{1}{2} e^{x}\leq \cosh(x)\leq e^{x}$,
$|\cos(x)|\leq 1$ and $g(x)\leq c\, x^p$, we obtain that
 \begin{align}
   \int_0^\infty \int_0^\infty
 \sup_{t\geq1}\Big|&\frac{e^{\frac{\pi^2}{8t}}g(a)}{\sqrt{\pi^2 t}\sqrt{a}}
 \exp\Big(-a\big(\cosh(y)\big)^2-\frac{y^2}{2t}\Big)
   \cosh(y)\cos\Big(\frac{\pi y}{2t}\Big)\Big|\,dy \,da \nonumber \\
&\leq \int_0^\infty \int_0^\infty c\, e^{\frac{\pi^2}{8}}a^{p-\frac{1}{2}}
 \exp\Big(-\frac{a e^{2y}}{4}\Big)
   \exp(y)\,dy \,da .\label{domcrit0}
 \end{align}
Using the substitution $y:=ax$, we see that
\begin{align} \label{eq:subsi}
 \int_0^{\infty} a^{p-\frac{1}{2}}e^{-ax}\,da
 =\int_0^{\infty} y^{p-\frac{1}{2}}x^{\frac{1}{2}-p}e^{-y}\frac{1}{x}\,dy
 = \frac{1}{x^{\frac{1}{2}+p}}\Gamma\left(p+\tfrac{1}{2}\right)
\end{align}
for every $x>0$.
Now applying Fubini's theorem and~\eqref{eq:subsi} yields that
\begin{align}
&\int_0^\infty \int_0^\infty a^{p-\frac{1}{2}}
 \exp\Big(-\frac{a e^{2y}}{4}\Big)\exp(y)\,dy \,da
 =\int_0^\infty \int_0^\infty  a^{p-\frac{1}{2}}
 \exp\Big(-a\frac{ e^{2y}}{4}\Big)\,da\exp(y) \,dy \nonumber \\
 &=\int_0^\infty \Gamma\left(p+\tfrac{1}{2}\right)\left(\frac{4}{e^{2y}}\right)^{p+\frac{1}{2}}e^y\,dy
 =\Gamma\left(p+\tfrac{1}{2}\right)2^{2p+1}\int_0^\infty e^{-2p y}\,dy\nonumber \\
& =\Gamma\left(p+\tfrac{1}{2}\right)2^{2p+1}\frac{1}{2p}
    <\infty.
    \label{domcrit}
 \end{align}
Thus the integrals in (\ref{domcrit0}) are finite.
Applying the dominated convergence theorem, we get that
\begin{align}
\lim_{t\rightarrow\infty}& \sqrt{t} \cdot \E\Big[g\big(\frac{1}{2A_{t}}\big)\Big]
= \lim_{t\rightarrow\infty} \sqrt{t}\int_0^\infty g(a) p_t(a)da \nonumber \\
&= \int_0^\infty g(a)\int_0^\infty\lim_{t\rightarrow\infty}
  \frac{\sqrt{2}e^{\frac{\pi^2}{8t}}}{\sqrt{\pi^2}}\frac{1}{\sqrt{a}}
 \exp\Big(-a\big(\cosh(y)\big)^2-\frac{y^2}{2t}\Big)
   \cosh(y)\cos\Big(\frac{\pi y}{2t}\Big)\,dy \,da \nonumber \\
&= \int_{0}^\infty g(a)
 \int_0^\infty\frac{\sqrt{2}}{\pi}\frac{1}{\sqrt{a}}\exp\Big(-a\big(\cosh(y)\big)^2\Big)
   \cosh(y)\,dy \,da .
\end{align}
To complete the lemma, we will simplify the inner integral in the above equation.
Noting that $(\cosh(y))^2=1+(\sinh(y))^2$, $y\in\mathbb{R}$,
and using the substitution $x:=\sqrt{2a}\sinh(y)$
 yields  that
\begin{align}
\begin{split}\nonumber
 \int_0^\infty\frac{\sqrt{2}}{\pi}\frac{1}{\sqrt{a}}\exp\Big(-a\big(\cosh(y)\big)^2\Big)
   \cosh(y)\,dy
 &=\int_0^\infty\frac{\sqrt{2}}{\pi}\frac{1}{\sqrt{a}}\exp\big(-a\big(1+(\sinh(y))^2\big)\big)
   \cosh(y)dy\\
 =\int_0^\infty\frac{\sqrt{2}}{\pi}\frac{1}{\sqrt{a}}\exp\Big(-a-\frac{x^2}{2}\Big)
   \frac{1}{\sqrt{2a}}\,dx
 &
   =\frac{1}{\sqrt{2\pi}}\frac{e^{-a}}{a}
\end{split}
\end{align}
for all $a\in(0,\infty)$.
\end{proof}

\begin{proof}[{Proof of the critical case of Theorem \ref{asymptotic_survival}}]
Fix $z\geq 0$. As $f(zx)\leq \frac{z\sigma_e^2}{\sigma_b^2}x$ for every $x\geq 0$, the assumptions
of Lemma~\ref{l:convergencecrit} are met with $g=f$.
Applying Corollary \ref{c:survival_prob} and Lemma \ref{l:convergencecrit}, we get that
 \begin{align}  
   \lim_{t\to\infty}\sqrt{t}\cdot\P^z\big(Z_t>0\big)
   &=\frac{2}{\sigma_e}\lim_{t\to\infty}\sqrt{\frac{t\sigma_e^2}{4}}\cdot
   \E\Big[f\Big(\frac{z}{2A_{t\sigma_e^2/4}}\Big)\Big]=\frac{2}{\sigma_e} \int_0^\infty f(za) \frac{1}{\sqrt{2\pi}}\frac{e^{-a}}{a} da.
\end{align}
Next we have that
\begin{equation}
 \begin{split}
 \int_0^\infty  (1-e^{-c x})\frac{e^{-x}}{x} \, dx
 =\log(1+c)
 \end{split}
\end{equation}
for every $c\in[0,\infty)$ which can be checked by differentiating both sides.
Using this result with $c=\tfrac{\sigma_e^2}{\sigma_b^2}z$
and recalling the definition (\ref{deff}) of $f$, we get that
\begin{align}
   \lim_{t\to\infty}\sqrt{t}\cdot\P^z\big(Z_t>0\big)
   = \frac{2}{\sigma_e} \int_0^\infty \left(1-\exp\left(-\tfrac{\sigma_e^2}{\sigma_b^2}za\right)\right)
                                  \frac{1}{\sqrt{2\pi}}\frac{e^{-a}}{a} da
  = \frac{\sqrt{2}}{\sqrt{\pi}\sigma_e} \log\Big(1+\frac{\sigma_e^2}{\sigma_b^2} \cdot z\Big),
\end{align}
which proves Theorem \ref{asymptotic_survival} in the case $\al=0$.
\end{proof}

\subsection{The weakly subcritical regime}
Next, we turn to the case $-\sigma_e^2<\alpha<0$ and prove (\ref{thweak}).
\begin{lemma}\label{l:convergence}
Let $\gamma>0$ and 
let $g:[0,\infty)\rightarrow [0,\infty)$ be a Borel measurable function.
Assume
for some $b>\gamma/2$ and $c>0$
that
$g(x)\leq c\, x^{b}$
for every $x\geq 0$.
Then we get that 
\begin{align}
\lim_{t\rightarrow\infty} t^{3/2}e^{\gamma^2 t/2} \E\Big[g\big(\frac{1}{2A_{t}^{(\gamma)}}\big)\Big] = \int_0^\infty g(a) \phi_{\gamma}(a) da <\infty.
\end{align}
\end{lemma}
\begin{proof}
Let $\gamma>0$ be fixed and let $g$ be a function fulfilling the
conditions of the lemma with constants $b>\gamma/2$ and $c>0$.
Recall the density of $\tfrac{1}{2 A_t^{(\gamma)}}$, $t>0$,
from (\ref{densitity_general}).
With the substitution $u=as$, we get for $t\in(0,\infty)$ that
\begin{align}
t^{3/2}& e^{\gamma^2 t/2} \E\Big[g\big(\frac{1}{2A_{t}^{(\gamma)}}\big)\Big] \nonumber \\
& =\frac{te^{\pi^2/2t}}{\sqrt{2}\pi^2} \Gamma\Big(\frac{\gamma+2}{2}\Big)\int_0^\infty\int_0^\infty\int_0^\infty g(a) e^{-a} a^{-(\gamma+1)/2} e^{-\xi^2/2t} \Big(\frac{u}{a}\Big)^{(\gamma-1)/2}e^{-u}\nonumber \\
&\qquad \qquad \qquad \qquad \qquad \qquad \qquad \qquad  \cdot \frac{\sinh(\xi)\cosh(\xi)\sin(\pi \xi/t)}{(u/a+(\cosh(\xi))^2)^{(\gamma+2)/2}}d\xi \,\frac{du}{a}\,da\nonumber
\\
&= \frac{te^{\pi^2/2t}}{\sqrt{2}\pi^2} \Gamma\Big(\frac{\gamma+2}{2}\Big)\int_0^\infty\int_0^\infty\int_0^\infty g(a) e^{-a} a^{-\gamma/2} e^{-\xi^2/2t}  u^{(\gamma-1)/2}e^{-u}\nonumber \\
&\qquad\qquad \qquad \qquad \qquad\qquad \qquad \qquad   \cdot \frac{\sinh(\xi)\cosh(\xi)\sin(\pi \xi/t)}{(u+a(\cosh(\xi))^2)^{(\gamma+2)/2}}d\xi\, du\,da.\label{1108}
\end{align}
In order to apply the dominated convergence theorem, we will show that the
integrand on the right-hand side is dominated for $t\in[1,\infty)$ by an
integrable function.
As for $x\geq 0$, $|\sin(x)|\leq x$, $\sinh(x)\leq \cosh(x)\leq e^{x}$, $\cosh(x)\geq e^{x}/2$ and by assumption, $g(x)\leq c\, x^b$,
 there is a constant $d=d_{\gamma}>0$ such that
 \begin{align}
&\int_0^\infty\int_0^\infty\int_0^\infty \sup_{t\geq 1}\bigg[\frac{te^{\pi^2/2t}}{\sqrt{2}\pi^2} \Gamma\Big(\frac{\gamma+2}{2}\Big)g(a) e^{-a} a^{-\gamma/2} e^{-\xi^2/2t}  u^{(\gamma-1)/2}e^{-u}\nonumber \\
&\qquad\qquad \qquad \qquad \qquad \qquad  \cdot \frac{\sinh(\xi)\cosh(\xi)\sin(\pi \xi/t)}{(u+a(\cosh(\xi))^2)^{(\gamma+2)/2}}\bigg]d\xi\, du\,da \nonumber
\\
&\leq d   \int_0^\infty\int_0^\infty\int_0^\infty e^{-a} a^{b-\gamma/2} u^{(\gamma-1)/2}e^{-u} \frac{e^{2\xi} \xi}{(4u+ae^{2\xi})^{(\gamma+2)/2}}d\xi\, du\,da\nonumber \\
&\leq d\, \bigg[\int_1^\infty\int_0^\infty\int_0^\infty e^{-a} a^{b-\gamma/2}  u^{(\gamma-1)/2}e^{-u}
\frac{e^{-\xi\gamma} \xi}{a^{(\gamma+2)/2}} d\xi\, du\,da \nonumber\\ 
&\qquad + \int_0^1\int_0^\infty\int_0^\infty a^{b-\gamma/2} u^{(\gamma-1)/2}e^{-u}
\frac{e^{-\xi\gamma} \xi}{(4ue^{-2\xi}+a)^{(\gamma+2)/2}} d\xi\, du\,da \bigg]. \label{0608}
 \end{align}
The first summand on
the right-hand side is finite as $\tfrac{\gamma-1}{2}>-1$ and $\gamma>0$.
Recall $b-\gamma/2>0$.
Choose $0<\epsilon<\min(b-\gamma/2,\gamma/2)$.
Using $a^x\leq a^y$ for every $a\in[0,1]$ and $0\leq y\leq x$,  and applying Fubini's theorem,
we estimate
the second summand on the right-hand side of (\ref{0608}) as follows:
\begin{align}
&\int_0^1\int_0^\infty\int_0^\infty a^{b-\gamma/2} u^{(\gamma-1)/2}e^{-u}
\frac{e^{-\xi\gamma} \xi}{(4ue^{-2\xi}+a)^{(\gamma+2)/2}} d\xi\, du\,da \bigg]\nonumber \\
&\qquad \leq   \int_0^1\int_0^\infty\int_0^\infty a^{\epsilon} u^{(\gamma-1)/2}e^{-u}
\frac{e^{-\xi\gamma} \xi}{(4ue^{-2\xi}+a)^{(\gamma+2)/2}} d\xi\, du\,da \nonumber\\
&\qquad \leq   \int_0^\infty\int_0^\infty\int_0^1 (4u e^{-2\xi}+a)^{\epsilon}  u^{(\gamma-1)/2}e^{-u}
\frac{e^{-\xi\gamma} \xi}{(4ue^{-2\xi}+a)^{(\gamma+2)/2}} da\,du\,d\xi \nonumber \\
&\qquad =   \int_0^\infty\int_0^\infty u^{(\gamma-1)/2}e^{-u}e^{-\xi\gamma} \xi \int_0^1  
\frac{1}{(4ue^{-2\xi}+a)^{(\gamma+2)/2-\epsilon}} da\,du\,d\xi \nonumber \\
&\qquad = \frac{1}{\gamma/2-\epsilon}  \int_0^\infty\int_0^\infty   u^{(\gamma-1)/2}e^{-u} e^{-\xi\gamma} \xi
\Big(\frac{1}{(4ue^{-2\xi})^{\gamma/2-\epsilon}}-\frac{1}{(4ue^{-2\xi}+1)^{\gamma/2-\epsilon}}\Big)du\, d\xi \nonumber \\ 
&\qquad \leq \frac{4^{-\gamma/2+\epsilon}}{\gamma/2-\epsilon}    \int_0^\infty \int_0^{\infty} u^{-1/2+\epsilon}e^{-u}  e^{-\epsilon 2\xi} \xi\, du \;d\xi<\infty.
\end{align}
Thus, applying dominated convergence in (\ref{1108}) and
using $t\sin(\pi \xi/t)\stackrel{t\rightarrow\infty}{\longrightarrow}\pi \xi$
for every $\xi\in[0,\infty)$, we get that
\begin{align}
&\lim_{t\rightarrow\infty} t^{3/2} e^{\gamma^2 t/2} \E\Big[g\big(\frac{1}{2A_{t}^{(\gamma)}}\big)\Big] \nonumber \\
&=  \Gamma\Big(\frac{\gamma+2}{2}\Big)\int_0^\infty\int_0^\infty\int_0^\infty \lim_{t\rightarrow\infty}\frac{te^{\pi^2/2t}}{\sqrt{2}\pi^2} g(a) e^{-a} a^{-\gamma/2} e^{-\xi^2/2t}  u^{(\gamma-1)/2}e^{-u}\nonumber \\
&\qquad\qquad\qquad\qquad \qquad\qquad\cdot \frac{\sinh(\xi)\cosh(\xi)\sin(\pi \xi/t)}{(u+a(\cosh(\xi))^2)^{(\gamma+2)/2}}d\xi\, du\,da\nonumber \\
&=\frac{1}{\sqrt{2}\pi} \Gamma\Big(\frac{\gamma+2}{2}\Big)\int_0^\infty\int_0^\infty\int_0^\infty g(a) e^{-a} a^{-\gamma/2}    u^{(\gamma-1)/2}e^{-u}\frac{\sinh(\xi)\cosh(\xi)\xi}{(u+a(\cosh(\xi))^2)^{(\gamma+2)/2}}d\xi\, du\,da\nonumber \\
&= \int_0^{\infty} g(a) \phi_{\gamma}(a)\,da,
\end{align}
which is the claim of the lemma.
\end{proof}
\begin{proof}[{Proof of Theorem \ref{asymptotic_survival} in the weakly subcritical case}]
Assume $-\sigma_e^2<\alpha<0$. Let $z\geq0$ be fixed and $\beta= -2\alpha/\sigma_e^2$. By Corollary \ref{c:survival_prob}, we obtain that
\begin{align}
\lim_{t\rightarrow\infty} t^{3/2} e^{\frac{\alpha^2}{2\sigma_e^2} t} \cdot\P^z\big(Z_t>0\big)
=  \frac{8}{\sigma_e^3} \lim_{t\rightarrow\infty} \Big(\frac{\sigma_e^2 t}{4}\Big)^{3/2}
   e^{\rob{\frac{2\alpha}{\sigma_e^2}}^2\cdot\frac{\sigma_e^2 t}{4}\cdot\frac{1}{2}} \E\bigg[f\Big(\frac{z}{2A_{t\sigma_e^2/4}^{(\beta)}}\Big)\bigg]. \label{weakcrit1}
\end{align}
As for every $x\geq 0$, $f(zx) \leq \frac{z\sigma_e^2}{\sigma_b^2} x$ and $0<\beta<2$,
the assumptions of Lemma \ref{l:convergence}
with $g=f$, $\gamma=\beta$ and $b=1>\tfrac{\beta}{2}$ are fulfilled.
Applying Lemma \ref{l:convergence} to (\ref{weakcrit1}) proves Theorem \ref{asymptotic_survival} in the weakly subcritical case.
\end{proof}
The following lemma for the weakly subcritical case will be needed later.
Recall $\beta=-2\al/\sigma_e^2$ and $\vartheta$ from~\eqref{eq:vartheta}.
\begin{lemma} \label{l:vartheta}
  Assume $\sigma_b,\sigma_e\in(0,\infty)$ and $\al\in(-\sigma_e^2,0)$.
  Then $\vartheta\in\C^\infty\left([0,\infty),\R\right)$ and
  \begin{align}
    \lim_{z\to\infty}\frac{\vartheta(z)}{z^{\frac{\beta}{2}}\log(z)}
 &=\frac{2}{\beta}\frac{\sqrt{2\pi}}{\sigma_e^3 \sin\big(\pi\frac{\beta}{2}\big)}
   \left(\frac{\sigma_e^2}{\sigma_b^2}\right)^{\frac{\beta}{2}}
    \label{eq:vartheta.vartheta}
    \\
    \lim_{z\to\infty}\frac{z\vartheta^{'}(z)}{z^{\frac{\beta}{2}}\log(z)}
 &=\frac{\sqrt{2\pi}}{\sigma_e^3 \sin\big(\pi\frac{\beta}{2}\big)}
   \left(\frac{\sigma_e^2}{\sigma_b^2}\right)^{\frac{\beta}{2}}
    \label{eq:vartheta.prime}.
  \end{align}
\end{lemma}
\begin{proof}
  In order to prove $\vartheta\in\C^\infty\left([0,\infty),\R\right)$, note that
\begin{align}
  \int_0^\infty \Big|\frac{\del^n f(za)}{\del^n z}\Big|\phi_\beta(a)\,da
  &=\int_0^\infty \frac{\sigma_e^{2n}}{\sigma_b^{2n}}a^{n}\exp\left(-\frac{\sigma_e^2}{\sigma_b^2}za\right)\phi_\beta(a)\,da
   \leq \frac{\sigma_e^{2n}}{\sigma_b^{2n}}\int_0^\infty a^{n}\, \phi_\beta(a)\,da
   \label{eq:deriv.weakly}
\end{align}
for all $z\in[0,\infty)$ and all $n\in\N$.
The right-hand side of~\eqref{eq:deriv.weakly}
is finite according to Lemma~\ref{l:convergence} (note $1>\beta/2$).
The dominated convergence theorem thus implies that
$\vartheta\in\C^\infty\left([0,\infty),\R\right)$.
Moreover, we get that
\begin{equation}   \label{eq:derivatives.vartheta}
  \vartheta^{(n)}(z)
  =
  \frac{8}{\sigma_e^3}
  \int_0^\infty
  f^{(n)}(za)a^n\phi_\beta(a)\,da
\end{equation}
for all $z\in[0,\infty)$, $n\in\N_0$.
For proving~\eqref{eq:vartheta.vartheta},
we substitute
$b:=za$ and get that
  \begin{equation}  \begin{split}  \label{eq:vartheta.1}
    \frac{\vartheta(z)}{z^{\frac{\beta}{2}}\log(z)}
    &=
    \frac{1}{z^{\frac{\beta}{2}}\log(z)}
    \frac{8}{\sigma_e^3}\int_0^\infty f(za)\phi_\beta(a)\,da
   =\frac{8}{\sigma_e^3}\int_0^\infty f(b)
    \frac{1}{z^{\frac{\beta}{2}+1}\log(z)}
    \phi_\beta\left(\frac{b}{z}\right)\,db
  \end{split}     \end{equation}
  for all $z\in(0,\infty)$.
  Next we rewrite
  the definition~\eqref{defphibeta} of $\phi_\beta$ and see that
  \begin{equation}  \begin{split}  \label{eq:vartheta.2}
  \lefteqn{
    \frac{1}{z^{\frac{\beta}{2}+1}\log(z)}
    \phi_\beta\left(\frac{b}{z}\right)
  }\\
    &=\frac{1}{\sqrt{2}\pi}\Gamma\Big(\frac{\beta+2}{2}\Big)
     e^{-\frac{b}{z}}b^{-\beta/2}\int_0^\infty u^{(\beta-1)/2}e^{-u}
     \int_0^\infty
     \frac{\sinh(\xi)\cosh(\xi) \xi}{z\log(z)\left(u+\frac{b}{z}(\cosh(\xi)\right)^2)^{\frac{\beta+2}{2}}}d\xi\,du
  \end{split}     \end{equation}
  for all $b\in(0,\infty)$ and all $z\in(0,\infty)$.
  Using the substitution $x:=\cosh(\xi)\sqrt{b/(zu)}$ and noting that
  $dx/d\xi=\sinh(\xi)\sqrt{b/(zu)}$, we obtain that
  \begin{align}
     \int_0^\infty
     \frac{\sinh(\xi)\cosh(\xi) \xi}{z\log(z)\left(u+\frac{b}{z}(\cosh(\xi)\right)^2)^{\frac{\beta+2}{2}}}d\xi
     &=
     \frac{u}{bu^{\frac{\beta+2}{2}}}
     \int_0^\infty
     \frac{\cosh(\xi)\sqrt{\tfrac{b}{zu}}\cdot \xi \cdot \sinh(\xi)\sqrt{\tfrac{b}{zu}}}
          {\log(z)\left(1+\frac{b}{zu}(\cosh(\xi)\right)^2)^{\frac{\beta+2}{2}}}d\xi
     \nonumber
     \\
     &=
     \frac{1}{bu^{\frac{\beta}{2}}}
     \int_{\sqrt{\frac{b}{zu}}}^\infty
     \frac{x \arcosh\left(x\sqrt{\frac{zu}{b}}\right)}{\log(z)\left(1+x^2\right)^{\frac{\beta+2}{2}}}dx 
     \label{eq:vartheta.3}
  \end{align}
  for all $b\in(0,\infty)$, $u\in(0,\infty)$ and all $z\in(0,\infty)$.
  Note that $\arcosh(y)=\log(y+\sqrt{y^2-1})$ for all $y\in[1,\infty)$ and therefore $\arcosh(y)\leq\log(2y)$
  for all $y\in[1,\infty)$.
  Consequently, we have that
  $\arcosh(x\sqrt{\tfrac{zu}{b}})/\log(z)\leq(\log(x\sqrt{\tfrac{u}{b}})\vee 1)+\log(2)+\tfrac{\log(\sqrt{z})}{\log(z)}
  \leq 3(\log(x\sqrt{\tfrac{u}{b}})\vee 1)$
  for
  all $x\in(\sqrt{b/(zu)},\infty)$
  and
  all $z\in[3,\infty)$.
  Let us prove that we may apply the dominated convergence theorem. From~\eqref{eq:vartheta.1}, \eqref{eq:vartheta.2} and \eqref{eq:vartheta.3} 
we see that with
$c:=\tfrac{8}{\sigma_e^3}\tfrac{1}{\sqrt{2}\pi}\Gamma\big(\tfrac{\beta+2}{2}\big)\in(0,\infty)$,
 \begin{align}   
    \frac{\vartheta(z)}{z^{\frac{\beta}{2}}\log(z)}
     &= c \int_{0}^\infty f(b) e^{-\frac{b}{z}} b^{-\frac{\beta}{2}-1} \int_0^\infty u^{-\frac{1}{2}}e^{-u}
     \int_{\sqrt{\frac{b}{zu}}}^\infty
     \frac{x \arcosh\left(x\sqrt{\frac{zu}{b}}\right)}{\log(z)\left(1+x^2\right)^{\frac{\beta+2}{2}}}dx 
    \,du
    \,db
    \nonumber
    \\
    &\leq c\int_{0}^\infty f(b) b^{-\frac{\beta}{2}-1} \int_0^\infty u^{-\frac12}e^{-u}
     \int_0^\infty \frac{x}{\left(1+x^2\right)^{\frac{\beta+2}{2}}}3\left(\log\left(x\sqrt{\tfrac{u}{b}}\right)\vee 1\right)dx
    \,du
    \,db
    \nonumber
\end{align}
for all $z\in[3,\infty)$.
The right-hand side is finite as $\beta/2\in(0,1)$
and as $f(b)\leq (b\sigma_e^2/\sigma_b^2)\wedge 1$ for $b\in[0,\infty)$.
Thus we may apply the dominated convergence theorem and get that
\begin{equation}  \begin{split}  \label{eq:vartheta.5}
 \lefteqn{
    \lim_{z\to\infty}\frac{\vartheta(z)}{z^{\frac{\beta}{2}}\log(z)}
 }
 \\
     &= c \int_{0}^\infty f(b) b^{-\frac{\beta}{2}-1} \int_0^\infty u^{-\frac{1}{2}}e^{-u}
     \int_{0}^\infty
   \lim_{z\to\infty} e^{-\frac{b}{z}}\1_{\sqrt{\frac{b}{zu}}< x}
     \frac{ \arcosh\left(x\sqrt{\frac{zu}{b}}\right)}{\log(z)}
     \frac{x}{\left(1+x^2\right)^{\frac{\beta+2}{2}}}
    \,dx 
    \,du
    \,db
 \\
     &= c \cdot\int_{0}^\infty f(b) b^{-\frac{\beta}{2}-1}
    \,db
    \cdot
     \int_0^\infty u^{-\frac{1}{2}}e^{-u} \,du
     \cdot
     \int_{0}^\infty
    \frac{1}{2}
     \frac{x}{\left(1+x^2\right)^{\frac{\beta+2}{2}}}
    \,dx 
  \end{split}     \end{equation}
  For the last step we used that $\arcosh(z)/\log(z)\to 1$ as $z\to\infty$.
We will simplify the integrals in~\eqref{eq:vartheta.5}.
Using integration by parts, we obtain that
\begin{align}
\begin{split}\label{eq:rewrite.integral}
 \int_{0}^\infty& f(b) b^{-\beta/2-1} \,db
 = \int_{0}^\infty b^{-\beta/2-1} \Big(1-\exp\Big(-\frac{\sigma_e^2}{\sigma_b^2} b\Big)\Big) \,db 
=\frac{2\sigma_e^2}{\beta\sigma_b^2}
  \cdot
   \left(\frac{\sigma_b^2}{\sigma_e^2}\right)^{1-\frac{\beta}{2}}
   \Gamma\big(1-\tfrac{\beta}{2}\big).
\end{split}
\end{align}
Inserting~\eqref{eq:rewrite.integral}
into~\eqref{eq:vartheta.5}
and using $\Gamma(\tfrac12)=\sqrt{\pi}$, $\Gamma(1+\tfrac{\beta}{2})=\tfrac{\beta}{2}\Gamma(\tfrac{\beta}{2})$
and Euler's reflection principle $\Gamma(x)\Gamma(1-x)=\pi/\sin(\pi x)$ for $x\in(0,1)$,
we get that
\begin{equation*}  \begin{split}
 \lim_{z\to\infty}\frac{\vartheta(z)}{z^{\frac{\beta}{2}}\log(z)}
 =\frac{8}{\sigma_e^3} \frac{1}{\sqrt{2}\pi }\frac{\beta}{2}\Gamma\Big(\frac{\beta}{2}\Big)
   \cdot\frac{2\sigma_e^2}{\beta\sigma_b^2}
   \left(\frac{\sigma_b^2}{\sigma_e^2}\right)^{1-\frac{\beta}{2}}
   \Gamma\Big(1-\frac{\beta}{2}\Big)
   \cdot
   \Gamma\Big(\frac{1}{2}\Big)
   \cdot
   \frac{1}{2\beta}
 =\frac{2\sqrt{2\pi}}{\sigma_e^3\beta \sin\big(\pi\frac{\beta}{2}\big)}
   \left(\frac{\sigma_e^2}{\sigma_b^2}\right)^{\frac{\beta}{2}}
\end{split}     \end{equation*}
This proves~\eqref{eq:vartheta.vartheta}.
  Paralleling
  the above arguments
  yields~\eqref{eq:vartheta.prime}.
\end{proof}

The following lemma implies that
the additional drift term in~\eqref{eq:SDE.conditioned}
is strictly decreasing in the weakly subcritical regime.
\begin{lemma}  \label{l:decreasing}
  Assume $\sigma_b,\sigma_e\in(0,\infty)$ and $\al\in(-\sigma_e^2,\infty)$.
  Then the function
  $(0,\infty)\ni z\mapsto \sigma_e^2 z\vartheta^{'}(z)/\vartheta(z)\in\R$
  is strictly monotonic decreasing
  and
  satisfies $\lim_{z\to 0}\sigma_e^2 z\vartheta^{'}(z)/\vartheta(z)=\sigma_e^2$
  and $\lim_{z\to\infty}\sigma_e^2 z\vartheta^{'}(z)/\vartheta(z)=\max(-\al,0)$.
\end{lemma}
\begin{proof}
  In all regimes $\al\in\R$, we have that $\vartheta^{'}(0)\in(0,\infty)$ and $\vartheta(0)=0$
  and, therefore,
  \begin{equation}
    \lim_{z\to0}\sigma_e^2\frac{z\vartheta^{'}(z)}{\vartheta(z)}
    =
    \sigma_e^2\frac{\lim_{z\to0}\vartheta^{'}(z)}{\lim_{z\to0}\frac{\vartheta(z)-\vartheta(0)}{z}}
    =
    \sigma_e^2\frac{\vartheta^{'}(0)}{\vartheta^{'}(0)}
    =\sigma_e^2.
  \end{equation}
  In the supercritical regime $\al>0$, we see that
  \begin{equation*}
    \frac{z\vartheta^{'}(z)}{\vartheta(z)}
    =
    \frac{z\frac{\dup}{\dup z}\Big(1-\Big(1+\frac{\sigma_e^2}{\sigma_b^2}\cdot z\Big)^{-\frac{2\alpha}{\sigma_e^2}}\Big)}
         {1-\Big(1+\frac{\sigma_e^2}{\sigma_b^2}\cdot z\Big)^{-\frac{2\alpha}{\sigma_e^2}}}
    =
    \frac{2\al z}{\sigma_b^2+\sigma_e^2z}
    \frac{1}{\Big(1+\frac{\sigma_e^2}{\sigma_b^2}\cdot z\Big)^{\frac{2\alpha}{\sigma_e^2}}-1}
  \end{equation*}
  for all $z\in(0,\infty)$.
  The derivative hereof is strictly negative
  \begin{equation*}  \begin{split}
    \frac{\dup}{\dup z}
    \frac{z\vartheta^{'}(z)}{\vartheta(z)}
    &=
    \frac{2\al\sigma_b^2
         \left(\Big(1+\frac{\sigma_e^2}{\sigma_b^2}\cdot z\Big)^{\frac{2\alpha}{\sigma_e^2}}-1\right)
         -4\al^2 z\Big(1+\frac{\sigma_e^2}{\sigma_b^2}\cdot z\Big)^{\frac{2\alpha}{\sigma_e^2}}
         }{
          \left(\sigma_b^2+\sigma_e^2 z\right)^2
         \left(\Big(1+\frac{\sigma_e^2}{\sigma_b^2}\cdot z\Big)^{\frac{2\alpha}{\sigma_e^2}}-1\right)^2
         }
    \\
    &=
    \frac{\int_0^z
         4\al^2
         \Big(1+\frac{\sigma_e^2}{\sigma_b^2}\cdot y\Big)^{\frac{2\alpha}{\sigma_e^2}-1}
         \,dy
         -4\al^2
         \int_0^z \Big(1+\frac{\sigma_e^2}{\sigma_b^2}\cdot y\Big)^{\frac{2\alpha}{\sigma_e^2}}
               + y\frac{2\al}{\sigma_b^2}\Big(1+\frac{\sigma_e^2}{\sigma_b^2}\cdot y\Big)^{\frac{2\alpha}{\sigma_e^2}-1}
               \,dy
         }{
          \left(\sigma_b^2+\sigma_e^2 z\right)^2
         \left(\Big(1+\frac{\sigma_e^2}{\sigma_b^2}\cdot z\Big)^{\frac{2\alpha}{\sigma_e^2}}-1\right)^2
         }
    \\
    &\leq
    \frac{
         -4\al^2
         \int_0^z
                y\frac{2\al}{\sigma_b^2}\Big(1+\frac{\sigma_e^2}{\sigma_b^2}\cdot y\Big)^{\frac{2\alpha}{\sigma_e^2}-1}
               \,dy
         }{
          \left(\sigma_b^2+\sigma_e^2 z\right)^2
         \left(\Big(1+\frac{\sigma_e^2}{\sigma_b^2}\cdot z\Big)^{\frac{2\alpha}{\sigma_e^2}}-1\right)^2
         }
    <0
  \end{split}     \end{equation*}
  for all $z\in(0,\infty)$.
  Moreover, it is clear that $\lim_{z\to\infty}z\vartheta^{'}(z)/\vartheta(z)=0$.

  In the critical regime $\al=0$, we get that
  \begin{equation*}
    \frac{z\vartheta^{'}(z)}{\vartheta(z)}
    =
    \frac{\sigma_e^2z}{\sigma_b^2+\sigma_e^2 z}
    \frac{1}{\log\left(1+\frac{\sigma_e^2}{\sigma_b^2}z\right)}
  \end{equation*}
  for all $z\in(0,\infty)$.
  The derivative hereof is strictly negative
  \begin{equation*}  \begin{split}
    \frac{\dup}{\dup z}
    \frac{z\vartheta^{'}(z)}{\vartheta(z)}
    &=\frac{\sigma_b^2\sigma_e^2\log\left(1+\frac{\sigma_e^2}{\sigma_b^2}z\right)
            -\sigma_e^2\sigma_e^2z
           }
          {\left(\sigma_b^2+\sigma_e^2 z\right)^2\left(\log\left(1+\frac{\sigma_e^2}{\sigma_b^2}z\right)\right)^2}
    <0
  \end{split}     \end{equation*}
  for all $z\in(0,\infty)$ where we used the inequality $\log(1+x)<x$ for all $x\in(0,\infty)$.
  Moreover, it is clear that $\lim_{z\to\infty}z\vartheta^{'}(z)/\vartheta(z)=0$.

  For the rest of the proof, we assume that $\al\in(-\sigma_e^2,0)$.
  Recall $\beta=-2\al/\sigma_e^2$.
  Lemma~\ref{l:vartheta} implies that
  \begin{equation}  \begin{split}
    \lim_{z\to\infty}\sigma_e^2\frac{z\vartheta^{'}(z)}{\vartheta(z)}
    =\sigma_e^2
    \frac{ \lim_{z\to\infty}\frac{z\vartheta^{'}(z)}{z^{\frac{\beta}{2}}\log(z)}
    }{
           \lim_{z\to\infty}\frac{\vartheta(z)}{z^{\frac{\beta}{2}}\log(z)}
    }
    =\sigma_e^2
    \frac{
       \frac{\sqrt{2\pi}}{\sigma_e^3 \sin\big(\pi\frac{\beta}{2}\big)}
        \left(\frac{\sigma_e^2}{\sigma_b^2}\right)^{\frac{\beta}{2}}
    }{
       \frac{2}{\beta}\frac{\sqrt{2\pi}}{\sigma_e^3 \sin\big(\pi\frac{\beta}{2}\big)}
        \left(\frac{\sigma_e^2}{\sigma_b^2}\right)^{\frac{\beta}{2}}
    }
    =\sigma_e^2\frac{\beta}{2}=-\al.
  \end{split}     \end{equation}
  It remains to prove monotonicity of
  $(0,\infty)\ni z\mapsto z\vartheta^{'}(z)/\vartheta(z)$.
  The first derivative hereof is
  \begin{equation}  \begin{split}  \label{eq:rewrite.derivative}
    \frac{\dup}{\dup z}
    \frac{z\vartheta^{'}(z)}{\vartheta(z)}
    &=
    \frac{\left(\frac{\beta}{2}\vartheta(z)-z\vartheta^{'}(z)\right)\vartheta^{'}(z)
           +
           \vartheta(z)\left(\left(1-\frac{\beta}{2}\right)\vartheta^{'}(z)+z\vartheta^{''}(z)\right)
    }{\left(\vartheta(z)\right)^2}
  \end{split}     \end{equation}
  for all $z\in(0,\infty)$.
  We will show that the right-hand side is negative for all $z\in(0,\infty)$.
  Define a function $\phit_{\beta}\colon(0,\infty)\to[0,\infty)$ through
  $\phit_{\beta}(a):=\tfrac{8}{\sigma_e^3}\phi_{\beta}(a)a^{\frac{\beta}{2}+1}$
  for all $a\in(0,\infty)$.
  Using the substitution $\eta=\sqrt{a}\cosh(\xi)$, we rewrite this function
  as
  \begin{equation*}  \begin{split}
    \phit_{\beta}(a)
    &=\frac{8}{\sigma_e^3}
      \int_0^\infty \int_0^\infty \frac{1}{\sqrt{2}\pi}\Gamma\Big(\frac{\beta+2}{2}\Big)
             e^{-a} u^{(\beta-1)/2}e^{-u} \frac{\sinh(\xi)\cosh(\xi) \xi}{(u+a(\cosh(\xi))^2)^{\frac{\beta+2}{2}}}d\xi\,du
      \cdot a
    \\
    &=\frac{8}{\sigma_e^3}\frac{1}{\sqrt{2}\pi}\Gamma\Big(\frac{\beta+2}{2}\Big)
     e^{-a}
      \int_0^\infty
        u^{(\beta-1)/2}e^{-u}
      \int_{\sqrt{a}}^\infty 
     \frac{\sinh(\xi)\frac{\eta}{\sqrt{a}}\arcosh\left(\frac{\eta}{\sqrt{a}}\right)}
          {(u+\eta^2)^{\frac{\beta+2}{2}}}\frac{d\eta}{\sqrt{a}\sinh(\xi)}\,du
      \cdot a
    \\
    &=\frac{8}{\sigma_e^3}\frac{1}{\sqrt{2}\pi}\Gamma\Big(\frac{\beta+2}{2}\Big)
     e^{-a}
      \int_0^\infty
        u^{(\beta-1)/2}e^{-u}
      \int_{0}^\infty 
      \1_{\eta>\sqrt{a}}\arcosh\left(\frac{\eta}{\sqrt{a}}\right)
     \frac{\eta}
          {(u+\eta^2)^{\frac{\beta+2}{2}}}d\eta\,du
  \end{split}     \end{equation*}
  for all $a\in(0,\infty)$. As $(0,\infty)\ni a\mapsto e^{-a}\arcosh\big(\tfrac{\eta}{\sqrt{a}}\big)\1_{\eta>\sqrt{a}}$
  is strictly monotonic decreasing for every $\eta\in(0,\infty)$,
  we conclude that $\phit_{\beta}$ is 
  strictly monotonic decreasing and that $\phit_{\beta}^{'}(a)<0$ for all $a\in(0,\infty)$.
  Moreover $\phit_{\beta}(a)$ decays at least exponentially fast as $a\to\infty$
  so that $\int_a^\infty b^{p}\phit_{\beta}(b)\,db<\infty$ for all $a,p\in(0,\infty)$.
  Applying~\eqref{eq:derivatives.vartheta} and integration by parts twice, it follows that
  \begin{equation}  \begin{split}  \label{eq:vartheta.esti2}
    \vartheta^{(n)}(z)
    &=\frac{8}{\sigma_e^3}\int_0^\infty f^{(n)}(za)a^n\phi_{\beta}(a)\,da
    =\int_0^\infty f^{(n)}(za)a^{n-\frac{\beta}{2}-1}\phit_{\beta}(a)\,da
    \\
    &=-\left[f^{(n)}(za)\int_{a}^\infty b^{n-\frac{\beta}{2}-1}\phit_{\beta}(b)\,db\right]_{0}^\infty
    +
    \int_0^\infty f^{(n+1)}(za)z\int_{a}^\infty b^{n-\frac{\beta}{2}-1}\phit_{\beta}(b)\,db\,da
    \\
    &=
    \int_0^\infty f^{(n+1)}(za)z
      \left(\Big[\frac{b^{n-\frac{\beta}{2}}}{n-\frac{\beta}{2}}\phit_{\beta}(b)\Big]_{a}^\infty
      -
      \int_{a}^\infty \frac{b^{n-\frac{\beta}{2}}}{n-\frac{\beta}{2}}\phit^{'}_{\beta}(b)\,db\right)\,da
    \\
    &=
    -\frac{1}{n-\frac{\beta}{2}}
    \int_0^\infty f^{(n+1)}(za)z
      a^{n-\frac{\beta}{2}}\phit_{\beta}(a)
      \,da
      -
    z\int_0^\infty f^{(n+1)}(za)
      \int_{a}^\infty \frac{b^{n-\frac{\beta}{2}}}{n-\frac{\beta}{2}}\phit^{'}_{\beta}(b)\,db\,da
  \end{split}     \end{equation}
  for all $z\in(0,\infty)$ and all $n\in\N_{0}$.
  Recall that $\phit_{\beta}^{'}(b)<0$ for all $b\in(0,\infty)$
  and note that $f^{(1)}(za)\frac{1}{-\frac{\beta}{2}}\phit_{\beta}^{'}(b)>0$
  for all $a,b,z\in(0,\infty)$.
  Thus~\eqref{eq:vartheta.esti2} with $n=0$ implies that
  \begin{equation}  \begin{split}  \label{eq:inequality.vartheta.1}
    \frac{\beta}{2}
    \vartheta(z)
    <
    -\frac{\beta}{2}
    \frac{1}{-\frac{\beta}{2}}
    \int_0^\infty f^{(1)}(za)za^{-\frac{\beta}{2}}\phit_{\beta}(a)\,da
    =
    z
    \frac{8}{\sigma_e^3}
    \int_0^\infty f^{(1)}(za)a\phi_{\beta}(a)\,da
    =
    z\vartheta^{'}(z)
  \end{split}     \end{equation}
  for all $z\in(0,\infty)$.
  Moreover
   $f^{(2)}(za)\frac{1}{1-\frac{\beta}{2}}\phit_{\beta}^{'}(b)>0$
  for all $a,b,z\in(0,\infty)$.
  So~\eqref{eq:vartheta.esti2} with $n=1$ yields that
  \begin{equation}  \begin{split}  \label{eq:inequality.vartheta.2}
    \left(1-\tfrac{\beta}{2}\right)
    \vartheta^{'}(z)
    <
    -\int_0^\infty f^{(2)}(za)za^{1-\frac{\beta}{2}}\phit_{\beta}(a)\,da
    =
    -z
    \frac{8}{\sigma_e^3}
    \int_0^\infty f^{(2)}(za)a^{2}\phi_{\beta}(a)\,da
    =
    -z
    \vartheta^{''}(z)
  \end{split}     \end{equation}
  for all $z\in(0,\infty)$.
  Applying the inequalities~\eqref{eq:inequality.vartheta.1}
  and~\eqref{eq:inequality.vartheta.2} to the right-hand side of~\eqref{eq:rewrite.derivative}
  and using $\vartheta^{'}(z),\vartheta(z)>0$ for all $z\in(0,\infty)$,
  we conclude that $\tfrac{\dup}{\dup z}\big(z\vartheta^{'}(z)/\vartheta(z)\big)<0$
  for all $z\in(0,\infty)$
  which implies that
  $(0,\infty)\ni z\mapsto z\vartheta^{'}(z)/\vartheta(z)$ is strictly
  monotonic decreasing.
\end{proof}

\subsection{The intermediately subcritical regime}
In this subsection, we will prove (\ref{thinter}).
\begin{proof}[{Proof of Theorem \ref{asymptotic_survival} in the intermediately subcritical case}]
Let $z\geq 0$ be fixed for the moment.
Lemma \ref{lem_hilf_intstrong} asserts that $\E\big[\tfrac{1}{2A_t^{(2)}}\big] = e^{-2t} \E\big[\tfrac{1}{2A_{t}}\big]$ and 
$\E\big[\tfrac{1}{(2A_t^{(2)})^2}\big] \leq e^{-2t} \rob{\E\big[\tfrac{1}{2A_{t/2}}\big]}^2$ for $t\in(0,\infty)$.
According to Lemma \ref{l:convergencecrit}, we have that
$\lim_{t\rightarrow\infty}\sqrt{t}\, \E\big[\tfrac{1}{2A_t}\big]
=\int_0^{\infty} a  \tfrac{1}{\sqrt{2\pi}}\tfrac{e^{-a}}{a}\,da<\infty$
and, therefore, $\big(\E\big[\tfrac{1}{2A_{t/2}}\big]\big)^2$ 
decays like $t^{-1}$ as $t\to\infty$.
Thus the assumptions of Lemma \ref{hilf} are met with $c_t=\sqrt{t}\, e^{2t}$
and $Y_t=1/\rob{2A_t^{(2)}}$, $t\geq1$.
Applying Corollary \ref{c:survival_prob} and Lemma \ref{hilf}, we get that
 \begin{align}
   \lim_{t\rightarrow\infty}\sqrt{t}\, e^{\frac{\sigma_e^2}{2} t} \cdot \P^z(Z_t>0)
   &= \sqrt{\frac{4}{\sigma_e^2}}\lim_{t\rightarrow\infty} \sqrt{\frac{\sigma_e^2 t}{4}}e^{2\frac{\sigma_e^2 t}{4}}
      \E\Big[f\Big(\frac{z}{2A^{(2)}_{t\sigma_e^2/4}}\Big)\Big]
 =\frac{2}{\sigma_e}\cdot\frac{z\sigma_e^2}{\sigma_b^2} \int_0^{\infty} a \frac{1}{\sqrt{2\pi}}\frac{e^{-a}}{a}\,da=z\,
 \frac{\sqrt{2}\sigma_e}{\sqrt{\pi}\sigma_b^2}, \nonumber
 \end{align}
which proves Theorem \ref{asymptotic_survival} in the case $\al=-\sigma_e^2$.
\end{proof}

\subsection{The strongly subcritical regime}
Finally, we will prove (\ref{thstrong}).
%
\begin{proof}[{Proof of Theorem \ref{asymptotic_survival} in the strongly subcritical case}]
Let  $\beta=-2\alpha/\sigma_e^2$ and assume $\alpha<-\sigma_e^2$.
Let $t>0$ and $z\in[0,\infty)$ be fixed.
The main tool for the proof is Lemma \ref{hilf}.
Let us check the conditions of that lemma.
Using Lemma \ref{lem_hilf_intstrong}, we obtain for the first and second moment of $1/(2A_t^{(\beta)})$ that
\begin{align}
\E\Big[\frac{1}{2A_t^{(\beta)}}\Big]&= e^{-(2\beta-2)t}\E\Big[\frac{1}{2A_t^{(-(\beta-2))}}\Big]\label{strongconv1}
\\
\E\Big[\frac{1}{(2A_t^{(\beta)})^2}\Big]&\leq e^{-(2\beta-2)t} \E\Big[\frac{1}{2A_{t/2}^{(\beta-2)}}\Big]\cdot \E\Big[\frac{1}{2A_{t/2}^{(-(\beta-2))}}\Big]
\end{align}
for every $t\in(0,\infty)$.
As $\beta>2$,
the monotone convergence theorem
and Lemma~\ref{lem_dufresne}
yield that
\begin{align}
\lim_{t\rightarrow\infty} \E\Big[\frac{1}{2A_{t/2}^{(-(\beta-2))}}\Big]
=
 \E\Big[\frac{1}{2A_{\infty}^{(-(\beta-2))}}\Big]
=\E\Big[ G_{\beta-2}\Big]
=\beta-2
<\infty. \label{strongconv2_a}
\end{align}
Using relation (1.1) from \cite{MatsumotoYor2003} and monotone convergence, we get that
\begin{align}   \label{strongconv2_b}
 \lim_{t\rightarrow\infty} \E\Big[\frac{1}{2A_{t/2}^{(\beta-2)}}\Big]
 =
 \lim_{t\rightarrow\infty} \E\Big[\frac{1}{2A_{t/2}^{(-(\beta-2))}}\Big]
   -\mathbb{E}\Big[G_{\beta-2}\Big]
 =0.
\end{align}
Thus, by (\ref{strongconv1}), (\ref{strongconv2_a})  and (\ref{strongconv2_b}),
the assumptions of Lemma \ref{hilf} are met with $c_t:= e^{(2\beta-2)t}$ and $Y_t=1/\rob{2A_t^{(\beta)}}$, $t\geq1$.
By Corollary \ref{c:survival_prob}, Lemma \ref{hilf} and \eqref{strongconv2_a}, we obtain that
\begin{equation}  \begin{split} \label{strongend}
  \lim_{t\rightarrow\infty} e^{(2\beta-2)\sigma_e^2 t/4}\cdot \P^z(Z_t>0)
  & = \lim_{t\rightarrow\infty} e^{(2\beta-2)\sigma_e^2 t/4}\cdot 
    \E\Big[f\Big(\frac {z}{2A_{t\sigma_e^2/4}^{(\beta)}}\Big)\Big]
  \\
  & = \frac{z\sigma_e^2}{\sigma_b^2}\lim_{t\rightarrow\infty} 
     \E\Big[\frac {1}{2A_{t\sigma_e^2/4}^{(-(\beta-2))}}\Big]
= \frac{z\sigma_e^2}{\sigma_b^2}(\beta-2)
= z2\frac{-\al-\sigma_e^2}{\sigma_b^2}
\end{split}     \end{equation}
Note that  $(2\beta-2)\sigma_e^2 /4=-(\alpha+\frac{\sigma_e^2}{2})$.
Thus, (\ref{strongend}) is the claim of Theorem \ref{asymptotic_survival} for the strongly subcritical case.
\end{proof}

%
%
%
\section{Proof of Lemma~\ref{l:conditioning.general}}
\label{sec:conditioning.general}
\begin{proof}[Proof of Lemma~\ref{l:conditioning.general}]
  Fix $t\in[0,\infty)$ and $x\in \bar{I}$ for the moment.
  As $\eta(x)>0$, there exists a $T_0\in[0,\infty)$ such that $\P^x(X_T\not\in A)>0$
  for all $T\in[T_0,\infty)$.
  Let $\phi\colon\C\left([0,t],I\right)\to\R$
  be a bounded and Borel measurable function.
  The Markov property of $(X_s)_{s\geq0}$ implies that
  \begin{equation}  \begin{split}
    \E^x\left[ \phi\left((X_s)_{s\leq t}\right) \1_{X_{T+t}\not\in A} \right]
    &=\E^x\left[ \phi\left((X_s)_{s\leq t}\right)
        \E^x\left[ \1_{X_{T+t}\not\in A}\Big|(X_s)_{s\leq t} \right]
         \right]
    \\
    &=\E^x\Big[ \phi\left((X_s)_{s\leq t}\right) \P^{X_t}\left(X_{T}\not\in A\right) \Big]
  \end{split}     \end{equation}
  for all $T\in[0,\infty)$.
  Consequently, we get for the conditional expectation that
  \begin{equation}  \begin{split}
      \E^x\left[ \phi\left((X_s)_{s\leq t}\right) \Big|{X_{T+t}\not\in A} \right]
  &=
  \frac{
    \E^x\left[ \phi\left((X_s)_{s\leq t}\right) \P^{X_t}\left(X_{T}\not\in A\right) \right]
       }{
         \P^x\left(X_{T+t}\not\in A\right)
       }
  \\
  &=
    \frac{
    \E^x\left[ \phi\left((X_s)_{s\leq t}\right) q(T)e^{\ld T}\P^{X_t}\left(X_{T}\not\in A\right) \right]
         }{ q(T+t)e^{\ld(T+t)} \P^x\left(X_{T+t}\not\in A\right) }
    \cdot
    \frac{q(T+t)e^{\ld (T+t)}}{q(T)e^{\ld T}}
  \end{split}     \end{equation}
  for all $T\in[T_0,\infty)$.
  Due to $\sup_{T\in[1,\infty)}q(T)e^{\ld T}\P^{X_t}\left(X_T\not\in A\right)\leq c\left(1+\|X_t\|^p\right)$
  for some constant $c\in[0,\infty)$
  and due to $\E[\|X_t\|^p]<\infty$, we may apply the dominated convergence theorem
  and obtain that
  \begin{equation}  \begin{split}\nonumber
    \limT\E^x\left[ \phi\left((X_s)_{s\leq t}\right) \Big|{X_{T+t}\not\in A} \right]
  &=
    \frac{\E^x\left[ \phi\left((X_s)_{s\leq t}\right) \limT q(T)e^{\ld T}\P^{X_t}\left(X_{T}\not\in A\right) \right]
         }{
          \limT q(T+t)e^{\ld(T+t)} \P^x\left(X_{T+t}\not\in A\right)
          }
     \limT\frac{q(T+t)}{q(T)}e^{\ld t}
   \\
   &=
   \frac{\E^x\left[\phi\left((X_s)_{s\leq t}\right)\eta(X_t)\right]
        }{ e^{-\ld t}\eta(x) }.
  \end{split}     \end{equation}
  If $\phi\equiv 1$, then the left-hand side is equal to $1$ and, therefore,
  the right-hand side is equal to $1$.
  Consequently the right-hand side defines a probability distribution on $\C\left([0,t],\R^d\right)$
  for every $t\in[0,\infty)$.
  These probability distributions are consistent for $t\in[0,\infty)$.
  So Kolmogorov's extension theorem (e.g.~\cite{Klenke2008})
  implies existence of a stochastic process $(\Xb_t)_{t\geq0}$ having continuous sample paths
  and being uniquely determined by
  \begin{equation}     \label{eq:determines.semigroup.uniquely}
    \limT\E^x\left[ \phi\left((X_s)_{s\leq t}\right) \Big|{X_{T+t}\not\in A} \right]
    =
    \E^x\left[\phi\left((\Xb_s)_{s\leq t}\right) \right]
    =
   \frac{\E^x\left[\phi\left((X_s)_{s\leq t}\right)\eta(X_t)\right]
        }{ e^{-\ld t}\eta(x) }
  \end{equation}
  for all $t\in[0,\infty)$.
  This proves~\eqref{eq:conditioning.general}.
  Moreover if $t\in[0,\infty)$ and if $g\colon I\to[0,\infty)$ is a Borel measurable function,
  then~\eqref{eq:semigroup.Xbar} follows from~\eqref{eq:determines.semigroup.uniquely}
  with $\phi\left((X_s)_{s\leq t}\right):=\min\left(g(X_t),n\right)$, $n\in\N$,
  and from the monotone convergence theorem as $n\to\infty$.

  Next we identify the linear operator of the martingale problem solved by $(\Xb_t)_{t\geq0}$.
  Similar arguments as above imply that
  \begin{equation}  \begin{split}
    \eta(x)&=0=\lim_{T\to\infty}q(T+t)e^{\ld(T+t)}\P^x\left(X_{T+t}\not\in A\right)
    \\
    &=\lim_{T\to\infty}\frac{q(T+t)}{q(T)}e^{\ld t}
    \E^x\left[\lim_{T\to\infty}q(T)e^{\ld T}\P^{X_t}\left(X_{T}\not\in A\right)\right]
   =e^{\ld t}\E^x\left[\eta(X_t)\right]
  \end{split}     \end{equation}
  for all $x\in I\setminus\bar{I}$ and for all $t\in[0,\infty)$.
  Moreover, the Markov property of $(X_t)_{t\geq0}$,
  the relation
  \begin{equation}
   \E^x[\eta(X_t)]-\eta(x)=\eta(x)\left(e^{-\ld t}-1\right)
    =-\int_0^t \eta(x)\ld e^{-\ld s}\,ds
    =\int_0^t \E^x\left[-\ld \eta\left(X_s\right)\right]\,ds
  \end{equation}
  for $x\in I$ and $t\in[0,\infty)$
  and Proposition 4.1.7 of~\cite{EthierKurtz1986}
  imply that
  \begin{equation}
    d\eta(X_t)=-\ld\eta(X_t)\,dt +dM_t
  \end{equation}
  for $t\in[0,\infty)$ where $(M_t)_{t\geq0}$ is a suitable martingale.
  Let $g\colon I\to\R$ be bounded and twice continuously differentiable.
  It\^{o}'s lemma shows that
  \begin{equation}
    d g(X_t)=\left(\Gen g\right)(X_t)\,dt+\left(\nabla g\,\sigma\right)(X_t)\,dW_t
  \end{equation}
  for all $t\in[0,\infty)$ where
  $\Gen g:=\nabla g\,\mu+\tfrac{1}{2}\operatorname{tr}\left(\sigma^t\left(\nabla^{t}\nabla g\right)\sigma\right)$.
  Thus It\^{o}'s lemma and symmetry of $\sigma\sigma^t$ result in
  \begin{equation*}  \begin{split}
  \lefteqn{
    de^{\ld t}\eta(X_t)g(X_t)
  }\\
    &=\ld e^{\ld t}\eta(X_t)g(X_t)\,dt+e^{\ld t}\,d\eta(X_t)g(X_t)
    \\
    &=
    \ld e^{\ld t}\eta(X_t)g(X_t)\,dt+e^{\ld t}\eta(X_t)\,d g(X_t)+e^{\ld t}g(X_t)\,d \eta(X_t)
    +e^{\ld t}\left(\nabla\eta \sigma\sigma^t\nabla^t g\right)(X_t)\,dt
    \\
    &=
    e^{\ld t}\eta(X_t)\left(\Gen g\right)(X_t)\,dt
    +e^{\ld t}\left(\eta \nabla g\sigma\right)(X_t)\,dW_t
    +e^{\ld t}g(X_t)\,dM_t
    +e^{\ld t}\left(\nabla\eta \sigma\sigma^t\nabla^t g\right)(X_t)\,dt
  \end{split}     \end{equation*}
  for all $t\in[0,\infty)$.
  Taking expectations, we infer that
  \begin{equation}
    \E^x\left[e^{\ld t}\eta(X_t)g(X_t)\right]-\eta(x)g(x)
    =
    \int_0^t \E^x\left[e^{\ld u}\eta(X_u)\left(\Genb g\right)(X_u)\right]\,du
  \end{equation}
  for all $t\in[0,\infty)$ and all $x\in I$
  where $\Genb g:=\Gen g+\frac{1}{\eta}\nabla\eta\,\sigma\sigma^t\nabla^t g$.
  This implies  for the Markov process $(\Xb_t)_{t\geq0}$ that
  \begin{equation}
    \E^x\left[g(\Xb_t)\right]-g(x)
    =
    \int_0^t \E^x\left[\left(\Genb g\right)(\Xb_u)\right]\,du
  \end{equation}
  for all $t\in[0,\infty)$ and all $x\in \bar{I}$.
  Now Proposition 4.1.7 in~\cite{EthierKurtz1986} implies
  that $(\Xb_t)_{t\geq0}$ is a solution of the martingale problem for $\Genb$.
  Finally, Theorem V.20.1 of~\cite{RogersWilliams2} shows that $(\Xb_t)_{t\geq0}$
  is a weak solution of the SDE~\eqref{eq:Xbar}.
  This completes the proof.
\end{proof}

\section{The BDRE conditioned to never go extinct}
\label{sec:The BDRE conditioned to go never extinct}
\begin{proof}[Proof of Theorem~\ref{thm:conditioning.BDRE}]
Fix $\al\in\R$ and $\sigma_b,\sigma_e\in(0,\infty)$.
We will prove Theorem \ref{thm:conditioning.BDRE} by applying Lemma \ref{l:conditioning.general}
to the process $(X_t)_{t\geq 0}=(Z_t,S_t)_{t\geq 0}$ which has state space $I:=[0,\infty)\times\R$.
Define $\mu\colon I\to\R^2$ and $\sigma\colon I\to\R^{2\times 2}$ by
\begin{equation}
\mu(z,s)=\left(\begin{array}{c}(\alpha+\frac{1}{2}\sigma_e^2) z\\
                               \alpha
               \end{array}
         \right)\quad\text{and}\quad
\sigma(z,s)=\left(\begin{array}{cc} \sqrt{\sigma_b^2 z} &  \sigma_e z\\
                                     0 & \sigma_e
                  \end{array}
             \right)
\end{equation}
for all $(z,s)\in[0,\infty)\times\R$.
We set $A=\{0\}\times\mathbb{R}$.
Note that $\{(Z_t,S_t)\not\in A\}=\{Z_t>0\}$
for all $t\in[0,\infty)$.
Moreover, define $\eta\colon I\to[0,\infty)$ through $\eta(z,s):=\vartheta(z)$ for all $(z,s)\in[0,\infty)\times\R$.
We will check the assumptions of Lemma~\ref{l:conditioning.general} for the different regimes separately.
In all cases we have that
\begin{equation}
  \E^z[|Z_t|^2]+\E^s[|S_t|^2]\leq z^2\E^0[e^{2S_t}]+2s^2+2\E^0[S_t^2]
  =z^2e^{2\al t+\sigma_e^2 t}+2s^2+2\sigma_e^2 t+2\al^2 t^2<\infty
\end{equation}
for all $(z,s)\in[0,\infty)\times\R$ and all $t\in[0,\infty)$.

\paragraph{The supercritical regime}
Let $q\equiv 1\in\CQ$, $\ld=0$ and $p=0$.
Theorem~\ref{asymptotic_survival} implies that
\begin{equation}
  \lim_{t\rightarrow\infty}q(t)e^{\ld t}  \,\mathbb{P}^{(z,s)}(Z_t>0) = 
  \eta(z,s)
  =\vartheta(z)
  =1-\Big(1+\frac{\sigma_e^2}{\sigma_b^2}\cdot z\Big)^{-\frac{2\alpha}{\sigma_e^2}}
\end{equation}
for all $(z,s)\in[0,\infty)\times\R$.
The function $\eta$ is twice continuously differentiable and satisfies
$\eta(z,s)>0$ if and only if $(z,s)\in(0,\infty)\times\R$.
Moreover, it is clear that
  $q(t)e^{\ld t}\P^{(z,s)}\left(Z_t>0\right)\leq 1$
for all $(z,s)\in I$ and for all $t\in[0,\infty)$.

\paragraph{The critical regime}
Define $q\in\CQ$ through $q(t)=\sqrt{t}$ for $t\in[0,\infty)$ and let $\ld=0$ and $p=1$.
Theorem~\ref{asymptotic_survival} implies that
\begin{equation}
  \lim_{t\rightarrow\infty}q(t)e^{\ld t}\,  \mathbb{P}^{(z,s)}(Z_t>0) = 
  \eta(z,s)
  =\vartheta(z)= \frac{\sqrt{2}}{\sqrt{\pi}\sigma_e} \log\Big(1+\frac{\sigma_e^2}{\sigma_b^2} \cdot z\Big)
\end{equation}
for all $(z,s)\in[0,\infty)\times\R$ and thus $\eta$ is twice continuously differentiable.
Note that $\eta(z,s)>0$ if and only if $(z,s)\in(0,\infty)\times\R$.
Moreover, Corollary \ref{c:survival_prob} implies that
\begin{equation}  \begin{split}
  \frac{1}{1+\|(z,s)\|}\sup_{t\in[1,\infty)} q(t)e^{\ld t}
  \P^{(z,s)}\left(Z_t>0\right)
  &=
  \frac{1}{1+\|(z,s)\|}\sup_{t\in[1,\infty)} \sqrt{t}\,
   \E\Big[f\Big(\frac{z}{2A_{t\sigma_e^2/4}}\Big)\Big]
   \\
 &\leq
  \frac{z}{1+z}\sup_{t\in[1,\infty)} \sqrt{t}\, 
  \frac{\sigma_e^2}{\sigma_b^2}\,
   \E\Big[\frac{1}{2A_{t\sigma_e^2/4}}\Big]
\end{split}     \end{equation}
for all $(z,s)\in[0,\infty)\times\R$.
The right-hand side is finite according to Lemma~\ref{l:convergencecrit}
and is uniformly bounded in $(z,s)\in[0,\infty)\times\R$.

\paragraph{The weakly subcritical regime}
Define $q\in\CQ$ through $q(t)=\sqrt{t}^3$ for $t\in[0,\infty)$ and let $\ld=\tfrac{\al^2}{2\sigma_e^2}$ and $p=1$.
Theorem~\ref{asymptotic_survival} implies that
\begin{equation}
  \lim_{t\rightarrow\infty}q(t)e^{\ld t}\,  \mathbb{P}^{(z,s)}(Z_t>0) = 
  \eta(z,s)
  =\vartheta(z)
  =\frac{8}{\sigma_e^3} \int_0^\infty f(za)\phi_\beta(a)\,da
\end{equation}
for all $(z,s)\in[0,\infty)\times\R$.
The function $\eta$ is twice continuously differentiable according to
Lemma~\ref{l:vartheta}.
Note that $\eta(z,s)>0$ if and only if $(z,s)\in(0,\infty)\times\R$.
Moreover, Corollary \ref{c:survival_prob} implies that
\begin{equation*}  \begin{split}
  \frac{1}{1+\|(z,s)\|}\sup_{t\in[1,\infty)} q(t)e^{\ld t}
  \P^{(z,s)}\left(Z_t>0\right)
  &=
  \frac{1}{1+\|(z,s)\|}\sup_{t\in[1,\infty)}\sqrt{t}^3
  e^{\frac{\al^2}{2\sigma_e^2}t}
   \E\Big[f\Big(\frac{z}{2A_{t\sigma_e^2/4}^{(\beta)}}\Big)\Big]
   \\
 &\leq
  \frac{z}{1+z}
  \frac{8}{\sigma_e^3}
  \frac{\sigma_e^2}{\sigma_b^2}
  \sup_{t\in[1,\infty)}
  \sqrt{\frac{t\sigma_e^2}{4}}^3
  e^{\big(\frac{2\al}{\sigma_e^2}\big)^2\cdot\frac{t\sigma_e^2}{4}\cdot\frac{1}{2}}
   \E\Big[\frac{1}{2A_{t\sigma_e^2/4}^{(\beta)}}\Big]
\end{split}     \end{equation*}
for all $(z,s)\in[0,\infty)\times\R$.
The right-hand side is finite according to Lemma~\ref{l:convergence}
and is uniformly bounded in $(z,s)\in[0,\infty)\times\R$.

\paragraph{The intermediately subcritical regime}
Define $q\in\CQ$ through $q(t)=\sqrt{t}$ for $t\in[0,\infty)$ and let $\ld=\tfrac{\sigma_e^2}{2}$ and $p=1$.
Theorem~\ref{asymptotic_survival} implies that
\begin{equation}
  \lim_{t\rightarrow\infty}q(t)e^{\ld t}\,  \mathbb{P}^{(z,s)}(Z_t>0) = 
  \eta(z,s)
  =\vartheta(z)
  =z\,\frac{\sqrt{2}\sigma_e}{\sqrt{\pi}\sigma_b^2}
\end{equation}
for all $(z,s)\in[0,\infty)\times\R$.
The function $\eta$ is twice continuously differentiable and satisfies
$\eta(z,s)>0$ if and only if $(z,s)\in(0,\infty)\times\R$.
Corollary \ref{c:survival_prob} implies that
\begin{equation*}  \begin{split}
  \frac{1}{1+\|(z,s)\|}\sup_{t\in[1,\infty)} q(t)e^{\ld t}
  \P^{(z,s)}\left(Z_t>0\right)
  &=
  \frac{1}{1+\|(z,s)\|}\sup_{t\in[1,\infty)}\sqrt{t}\, e^{\frac{\sigma_e^2}{2}t}
   \E\Big[f\Big(\frac{z}{2A_{t\sigma_e^2/4}^{(2)}}\Big)\Big]
   \\
 &\leq
  \frac{z}{1+z}
  \frac{\sigma_e^2}{\sigma_b^2}
  \sqrt{\frac{4}{\sigma_e^2}}
  \sup_{t\in[1,\infty)}
  \sqrt{\frac{t\sigma_e^2}{4}}e^{2\frac{t\sigma_e^2}{4}}
   \E\Big[\frac{1}{2A_{t\sigma_e^2/4}^{(2)}}\Big]
\end{split}     \end{equation*}
for all $(z,s)\in[0,\infty)\times\R$.
The right-hand side is finite according to Lemma~\ref{lem_hilf_intstrong} with $\gamma=2$
and according to Lemma~\ref{l:convergencecrit},
and is uniformly bounded in $(z,s)\in[0,\infty)\times\R$.

\paragraph{The strongly subcritical regime}
Let $q\equiv 1\in\CQ$, $\ld=-\left(\al+\tfrac{\sigma_e^2}{2}\right)=(2\beta-2)\tfrac{\sigma_e^2}{4}$ and $p=1$.
Theorem~\ref{asymptotic_survival} implies that
\begin{equation}
  \lim_{t\rightarrow\infty}q(t)e^{\ld t}\,  \mathbb{P}^{(z,s)}(Z_t>0) = 
  \eta(z,s)
  =\vartheta(z)
  =z \cdot 2\frac{-\al-\sigma_e^2}{\sigma_b^2}
\end{equation}
for all $(z,s)\in[0,\infty)\times\R$.
The function $\eta$ is twice continuously differentiable and satisfies
$\eta(z,s)>0$ if and only if $(z,s)\in(0,\infty)\times\R$.
Corollary \ref{c:survival_prob} implies that
\begin{equation}  \begin{split}\label{eq:sup.strongly}
  \frac{1}{1+\|(z,s)\|}\sup_{t\in[1,\infty)} q(t)e^{\ld t}
  \P^{(z,s)}\left(Z_t>0\right)
  &=
  \frac{1}{1+\|(z,s)\|}\sup_{t\in[1,\infty)}e^{\ld t}
   \E\Big[f\Big(\frac{z}{2A_{t\sigma_e^2/4}^{(\beta)}}\Big)\Big]
\\
 \leq
  \frac{z}{1+z}
  \frac{\sigma_e^2}{\sigma_b^2}
  \sup_{t\in[1,\infty)}
  e^{(2\beta-2)\frac{t\sigma_e^2}{4} }
   \E\Big[\frac{1}{2A_{t\sigma_e^2/4}^{(\beta)}}\Big]
 &=
  \frac{z}{1+z}
  \frac{\sigma_e^2}{\sigma_b^2}
  \sup_{t\in[1,\infty)}
   \E\Big[\frac{1}{2A_{t\sigma_e^2/4}^{(-(\beta-2))}}\Big]
\end{split}     \end{equation}
for all $(z,s)\in[0,\infty)\times\R$.
The last step follows from Lemma~\ref{lem_hilf_intstrong}.
The right-hand side of~\eqref{eq:sup.strongly} is finite
due to $-(\beta-2)<0$ and due to Lemma~\ref{lem_dufresne}
and is uniformly bounded in $(z,s)\in[0,\infty)\times\R$.

\paragraph{Application of Lemma~\ref{l:conditioning.general}}
After having checked all assumptions, we
apply Lemma~\ref{l:conditioning.general}.
The additional drift term is
\begin{equation}  \begin{split}
  \frac{1}{\eta(z,s)}\left(\sigma\sigma^t\nabla^{t}\eta\right)(z,s)
  &=\frac{1}{\vartheta(z)}
\left(\begin{array}{cc} \sqrt{\sigma_b^2 z} &  \sigma_e z\\
                                     0 & \sigma_e
                  \end{array}
\right)
\left(\begin{array}{cc} \sqrt{\sigma_b^2 z} &  0\\
                                 \sigma_e z & \sigma_e
                  \end{array}
\right)
\left(\begin{array}{c} \vartheta^{'}(z)\\
                       0
                  \end{array}
\right)
\\
  &=\frac{1}{\vartheta(z)}
\left(\begin{array}{c} \sigma_b^2 z +\sigma_e^2 z^2 \\
                                 \sigma_e^2 z 
                  \end{array}
\right)
  \vartheta^{'}(z)
\end{split}     \end{equation}
for $(z,s)\in\bar{I}=(0,\infty)\times \R$.
Inserting this into~\eqref{eq:Xbar},
we get for $(\Zb_t,\Sb_t)_{t\geq0}$ that
\begin{equation*}  \begin{split}
  d\Zb_t&=\frac{\vartheta^{'}(\Zb_t)}{\vartheta(\Zb_t)}\Big(\sigma_b^2\Zb_t+\sigma_e^2\Zb_t^2\Big)\,dt
         +\left(\frac{1}{2}\sigma_e^2\Zb_t+\al\Zb_t \right)\,dt
        +\sqrt{\sigma_b^2\Zb_t}dW_t^{(b)}
         +\Zb_t\sigma_e dW_t^{(e)}
        \\
  d\Sb_t&=\frac{\vartheta^{'}(\Zb_t)}{\vartheta(\Zb_t)}\sigma_e^2\Zb_t\,dt+\al\,dt
          +\sigma_e dW_t^{(e)}.
\end{split}     \end{equation*}
Therefore $(\Zb_t,\Sb_t)_{t\geq0}$ solves the SDEs~\eqref{eq:SDE.conditioned}.
Moreover, Lemma~\ref{l:conditioning.general} implies that
the conditioned process satisfies~\eqref{eq:semigroup.Zb}.
In addition
Lemma~\ref{l:decreasing} establishes the properties of the
function $(0,\infty)\ni z\mapsto \sigma_e^2 z\vartheta^{'}(z)/\vartheta(z)$.

It remains to establish the limit of $\Zb_t$ as $t\to\infty$.
Note that $(\Zb_t)_{t\geq0}$ is a one-dimensional diffusion
with drift term $\mu(z):=\tfrac{\vartheta^{'}(z)}{\vartheta(z)}(\sigma_b^2 z+\sigma_e^2 z^2)+\tfrac{1}{2}\sigma_e^2 z+\al z$,
$z\in(0,\infty)$ and diffusion term $\sigma^2(z):=\sigma_b^2 z+\sigma_e^2 z^2$, $z\in(0,\infty)$.
Define a scale function $R\colon[0,\infty]\to[-\infty,\infty]$ through
\begin{equation}
  R(z):=\int_1^z\exp\left(-\int_1^y\frac{2\mu(z)}{\sigma^2(z)}\,dz\right)\,dy \label{eq:invariant.general}
\end{equation}
for all $z\in[0,\infty]$.
Standard results (e.g.~\cite{KaratzasShreve1991}) show that
$\Zb_t\to\infty$ in distribution as $t\to\infty$ if $R(0)=-\infty$ and $R(\infty)<\infty$.
Let $\al\in\R$.
We rewrite the  integral in the exponent
on the right-hand side of~\eqref{eq:invariant.general} as
\begin{equation}  \begin{split}  \label{eq:rewrite.logSprime}
  \int_1^y&\frac{2\frac{\vartheta^{'}(z)}{\vartheta(z)}(\sigma_b^2 z+\sigma_e^2 z^2)+\sigma_e^2 z+2\al z}{\sigma_b^2 z+\sigma_e^2 z^2}\,dz
 =2\int_1^y\frac{\vartheta^{'}(z)}{\vartheta(z)}\,dz+\int_1^y\frac{\sigma_e^2+2\al}{\sigma_b^2+\sigma_e^2 z}\,dz
 \\
 &=\log\left(\left(\vartheta(y)\right)^2\right)
   +\frac{\sigma_e^2+2\al}{\sigma_e^2}\log\left(\sigma_b^2 +\sigma_e^2 y\right)
   -\log\left(\left(\vartheta(1)\right)^2\right)
   -\frac{\sigma_e^2+2\al}{\sigma_e^2}\log\left(\sigma_b^2 +\sigma_e^2 \right)
\end{split}     \end{equation}
for all $y\in(0,\infty)$.
By (\ref{eq:rewrite.logSprime}), there exists a constant $c\in(0,\infty)$ such that
\begin{equation}
  R(z)=c\int_1^z\frac{1}{\left(\vartheta(y)\right)^2}\left(\sigma_b^2+\sigma_e^2 y\right)^{-\frac{\sigma_e^2+2\al}{\sigma_e^2}}
        \,dy
\end{equation}
for all $z\in[0,\infty]$.
As $\vartheta(z)/z\to \tilde{c}$ as $z\to0$
for a constant $\tilde{c}=\tilde{c}(\alpha,\sigma_e,\sigma_b)>0$, we have that
$ \lim_{z\to0} R(z)\approx \int_1^0 \frac{1}{y^2} \, dy =-\infty$. 
Next we show that $R(\infty)<\infty$ whenever $\al>-\sigma_e^2$.
If $\al>0$, then $\vartheta(z)=1-\left(1+\frac{\sigma_e^2}{\sigma_b^2}\cdot z\right)^{-\frac{2\al}{\sigma_e^2}}$ and thus $R(\infty)<\infty$.
In the case $\al=0$ (and thus $\beta=0$), we have that $\vartheta(z)=\frac{\sqrt{2}}{\sqrt{\pi} \sigma_e} \log(1+\frac{\sigma_e^2}{\sigma_b^2} z)$. From $\int_2^\infty \frac{1}{y(\log(y))^2}dy<\infty$, we deduce that $R(\infty)<\infty$.
Next let $\al\in(-\sigma_e^2,0)$. 
Lemma~\ref{l:vartheta} implies that there is a constant $\hat{c}\in(0,\infty)$ such that
\begin{equation}
  R(z)\sim \hat{c}\int_2^z\frac{1}{\left(y^{\beta/2}\log(y)\right)^2}\left(\sigma_b^2+\sigma_e^2 y\right)^{\beta -1}\,dy
\end{equation}
as $z\to\infty$.
As $\int_2^\infty\frac{1}{y(\log(y))^2}\,dy<\infty$, this implies that $R(\infty)<\infty$.
Finally assume that $\al<-\sigma_e^2$.
Theorem V.54.5 in~\cite{RogersWilliams2000b} implies
that
  \begin{equation}  \label{eq:invariant.general1}
    \P^z\left(\Zb_t\in dy\right)\varwlimt
    \bar{c}\,\frac{2}{\sigma^2(y)}\exp\left(\int_1^y\frac{2\mu(u)}{\sigma^2(u)}\,du\right)\,dy
  \end{equation}
for every $z\in(0,\infty)$
if there exists a normalizing constant $\bar{c}\in(0,\infty)$
such that the right-hand side is a probability distribution.
Due to~\eqref{eq:rewrite.logSprime} we need to show that
\begin{equation}  \begin{split}  \label{eq:integrable}
  \frac{1}{\sigma_b^2 y+\sigma_e^2 y^2}\left(\vartheta(y)\right)^2
       \left(\sigma_b^2+\sigma_e^2 y\right)^{\frac{\sigma_e^2+2\al}{\sigma_e^2}}
  =
  \left(2\frac{-\al-\sigma_e^2}{\sigma_b^2}\right)^2y\left(\sigma_b^2+\sigma_e^2 y\right)^{\frac{2\al}{\sigma_e^2}}
\end{split}     \end{equation}
is integrable over $y\in(0,\infty)$.
This function is bounded over $(0,1]$ and
is of order $O(y^{1+\frac{2\al}{\sigma_e^2}})$ as $y\to\infty$.
As $\al<-\sigma_e^2$, there exists a normalizing constant $\bar{c}$ such that
the right-hand side of~\eqref{eq:invariant.general1} is a probability distribution.
\end{proof}

%
%
%
\section{Family decomposition of BDREs with immigration}
\label{sec:family.decomposition}

Let $\al,\th\in\R$, $\sigma_b\in(0,\infty)$ and $\sigma_e\in[0,\infty)$.
In this section we consider the BDRE with immigration/emigration
which is the solution of the SDEs
\begin{equation}  \begin{split}  \label{eq:BDREI}
 dZ_t&=\th\,dt+\frac 12 \sigma_e^2 Z_t\, dt+Z_t dS_t+\sqrt{\sigma_b^2 Z_t}dW^{(b)}_t\\
 dS_t&=\alpha\, dt +\sqrt{\sigma_e^2}dW^{(e)}_t
\end{split}     \end{equation}
for $t\geq0$ where $S_0=0$.
The family decomposition of the BDRE with immigration will be a corollary
of the family decomposition of Feller's branching diffusion with immigration.
For this, we first need to generalize Proposition~\ref{p:time.change}
to include immigration.
\begin{lemma}   \label{l:time.change.immigration}
  Assume $\al,\th\in\R$, $\sigma_b\in(0,\infty)$ and $\sigma_e\in[0,\infty)$.
  Let $(F_t)_{t\geq0}$ be a weak solution of
  \begin{equation}
    dF_t=\frac{\th}{\sigma_b^2}\,dt + \sqrt{F_t}dW_t^{(b)}
  \end{equation}
  for $t\in[0,\infty)$
  and let $S_t:=\al t+\sigma_e W_t^{(e)}$ for $t\in[0,\infty)$
  be independent of $(F_t)_{t\geq0}$.
  Moreover, define $(\tau(t))_{t\geq0}$ through
    $\tau(t):=\int_0^t e^{-S_s}\sigma_b^2\,ds$
  for $t\in[0,\infty)$.
  Then
  \begin{equation}  \label{eq:time.change.immigration}
    \left( F_{\tau(t)}e^{S_t}, S_t \right)_{t\geq 0}
  \end{equation}
  is a weak solution of~\eqref{eq:BDREI}.
\end{lemma}
\begin{proof}
  Fix $\al,\th\in\R$, $\sigma_b\in(0,\infty)$ and $\sigma_e\in[0,\infty)$.
  Define $Z_t:=F_{\tau(t)}e^{S_t}$ for $t\in[0,\infty)$.
  It\^{o}'s lemma together with independence of $(F_t)_{t\geq0}$
  and of $(S_t)_{t\geq0}$
  imply that
  \begin{equation}  \begin{split}  \label{eq:time.change.Ito}
    dZ_t&=d F_{\tau(t)}e^{S_t}\\
    &=e^{S_t}dF_{\tau(t)}
      +F_{\tau(t)}e^{S_t}d S_t
      +\frac{1}{2}F_{\tau(t)}e^{S_t}\sigma_e^2\,dt\\
    &=e^{S_t}\frac{\th}{\sigma_b^2}d\tau(t)+e^{S_t}\sqrt{F_{\tau(t)}}dW_{\tau(t)}^{(b)}
      +Z_td S_t
      +\frac{1}{2}Z_t\sigma_e^2\,dt\\
    &=e^{S_t}\frac{\th}{\sigma_b^2}\tau^{'}(t)\,dt
      +\sqrt{\sigma_b^2 F_{\tau(t)}e^{S_t}}\frac{1}{\sqrt{\sigma_b^2 e^{-S_t}}}dW_{\tau(t)}^{(b)}
      +Z_td S_t +\frac{1}{2}Z_t\sigma_e^2\,dt\\
    &=\th\,dt
      +\frac{1}{2}\sigma_e^2 Z_t\,dt
      +Z_td S_t
      +\sqrt{\sigma_b^2 Z_t}\frac{1}{\sqrt{\tau^{'}(t)}}dW_{\tau(t)}^{(b)}
  \end{split}     \end{equation}
  for $t\in[0,\infty)$.
  As $(\tau(t))_{t\geq0}$
  and $(W_t^{(b)})_{t\geq0}$ are independent,
  the process $(W_t)_{t\geq0}$ defined through
  $W_t:=\int_0^t\tfrac{1}{\sqrt{\tau^{'}(s)}}dW_{\tau(s)}^{(b)}$, $t\in[0,\infty)$,
  is a continuous martingale and a Markov process satisfying
  \begin{equation}
    \E\left[W_t^2\right]
    =
    \E\left[\left(\int_0^t\tfrac{1}{\sqrt{\tau^{'}(s)}}dW_{\tau(s)}^{(b)}\right)^2\right]
    =\E\left[\int_0^t\tfrac{1}{\tau^{'}(s)}d\tau(s)\right]
    =t
  \end{equation}
  for all $t\in[0,\infty)$.
  Thus $(W_t)_{t\geq0}$ is a standard Brownian motion according to
  L\'{e}vy's characterization (e.g.~Theorem IV.33.1 of~\cite{RogersWilliams2}).
  Moreover $(W_t)_{t\geq0}$ and
  $(W_t^{(e)})_{t\geq0}$ are independent.
  Therefore \eqref{eq:time.change.Ito} implies
  that $(Z_t,S_t)_{t\geq0}$ is a weak solution of~\eqref{eq:BDREI}.
\end{proof}
Let $\sigma_b\in(0,\infty)$ and let $(F_t)_{t\geq0}$ be the solution of the SDE
\begin{equation}  \label{eq:F.general}
  dF_t=\sqrt{\sigma_b^2 F_t}dW_t
\end{equation}
for $t\in[0,\infty)$.
Recall the associated excursion measure $Q_F$ on $U$ from~\eqref{eq:excursion.measure}.
The following family decomposition of Feller's branching diffusion
is a special case of the family decomposition of
the Dawson-Watanabe superprocess with immigration
(see~\cite{LiShiga1995} and \cite{Dawson1993}).
Recall the excursion space $U$ from~\eqref{eq:excursion.space}.
\begin{lemma}  \label{l:family.decomposition.Feller}
  Let $\th\in[0,\infty)$, $\al\in\R$ and $\sigma_b\in(0,\infty)$.
  Let $\CP_0$ be a Poisson point process
  on $[0,\infty)\times U$ with intensity measure $dy\times Q_F$
  and
  let $\tilde{\CP_\th}$ be an independent
  Poisson point process
  on $[0,\infty)\times U$ with intensity measure $\th dt\times Q_F$.
  Then the process $(\Ft_t)_{t\geq 0}$ defined through $\Ft_0=x$
  and
  \begin{equation}  \label{eq:family.decomposition.Feller}
    \Ft_t:=\sum_{(y,\chi)\in\CP_0}\1_{y\leq x}\,\chi_t
    +\sum_{(s,\chi)\in\tilde{\CP_\th}}\chi_{t-s}
  \end{equation}
  for $t\in(0,\infty)$ is a weak solution of the SDE
  \begin{equation}       \label{eq:SDE.immigration}
    d\Fb_t=\th\,dt+\sqrt{\sigma_b^2 \Fb_t}dW_t,\quad \Fb_0=x,
  \end{equation}
  for $t\in[0,\infty)$ and
  for each $x\in[0,\infty)$.
\end{lemma}

\begin{proof}[{Proof of Theorem~\ref{thm:family.decomposition}}]
  Fix $\sigma_b,\sigma_e\in(0,\infty)$ and $\al\in(-\infty,-\sigma_e^2]$.
  Define a process $(\Ft_t)_{t\geq0}$ through $\Ft_0=x$ and through
  \begin{equation}
    \Ft_t:=\sum_{(y,\chi)\in\CP}\1_{y\leq x}\,\chi_{t}
    +\sum_{(s,\chi)\in\tilde{\CP}}\chi_{t-s}
  \end{equation}
  for $t\in(0,\infty)$.
  Then Lemma~\ref{l:family.decomposition.Feller} shows that $(\Ft_t)_{t\geq0}$ is a weak solution of
  \begin{equation}
    d\Fb_t=dt+\sqrt{\Fb_t}dW_t
  \end{equation}
  for $t\in[0,\infty)$.
  Lemma~\ref{l:time.change.immigration} implies that
  $(\Zt_t,\St_t)_{t\geq0}=(\Ft_{\tilde{\tau}(t)}e^{\St_t},\St_t)_{t\geq0}$
  is a weak solution of the SDEs~\eqref{eq:Zb.linear}.
  Also $(\Zb_t,\Sb_t)_{t\geq0}$ is a solution of~\eqref{eq:Zb.linear}
  due to Theorem~\ref{thm:conditioning.BDRE}.
  As the solution of~\eqref{eq:Zb.linear} is unique in law,
  we conclude that
  $(\Zt_t,\St_t)_{t\geq0}$
  and
  $(\Zb_t,\Sb_t)_{t\geq0}$ have the same distribution.
\end{proof}

\def\cprime{$'$}
  \hyphenation{Sprin-ger}\def\cftil$1{\ifmmode\setbox7\hbox{$\accent"5E$1$}\el%
se \setbox7\hbox{\accent"5E$1}\penalty 10000\relax\fi\raise 1\ht7
  \hbox{\lower1.15ex\hbox to 1\wd7{\hss\accent"7E\hss}}\penalty 10000
  \hskip-1\wd7\penalty 10000\box7}

\end{document}